\documentclass[10pt,a4paper]{article}
   \usepackage{a4}
   \usepackage[french, english]{babel}
   \usepackage{amssymb,amsmath,amsfonts, amsthm}
   \usepackage{cite}
   \usepackage{graphicx}
   \usepackage[T1]{fontenc}
   \usepackage[arrow, matrix, curve]{xy}
   \usepackage{paralist}
   \usepackage[fleqn,tbtags]{mathtools}
   \usepackage{psfrag}

\theoremstyle{plain}
\newtheorem{thm}{Th\'eor\`eme}[section]
\newtheorem{corol}[thm]{Corollaire}
\newtheorem{lemme}[thm]{Lemme}
\newtheorem{prop}[thm]{Proposition}

\newtheorem{question}[thm]{Question}
\newtheorem{conjecture}[thm]{Conjecture}

\theoremstyle{definition}
\newtheorem{definition}[thm]{D\'efinition}
\newtheorem{rem}[thm]{Remarque}
\newtheorem{ex}[thm]{Exemple}
\newtheorem{nota}[thm]{Notation}
\newtheorem{nota-rem}[thm]{Notation-Remarque}

\DeclareMathOperator{\ext}{ext}

\DeclareMathOperator{\id}{id}

\DeclareMathOperator{\GL}{GL}
\DeclareMathOperator{\End}{End}

\DeclareMathOperator{\Nef}{Nef}
\DeclareMathOperator{\Psef}{Psef}
\DeclareMathOperator{\Semi}{Semi}

\DeclareMathOperator{\N}{N}

\DeclareMathOperator{\ZZ}{\mathbb{Z}}

\DeclareMathOperator{\rg}{rang}
\DeclareMathOperator{\extr}{ext}

\DeclareMathOperator{\Sym}{\mathbf{S}}
\DeclareMathOperator{\Symm}{Sym}

\DeclareMathOperator{\G}{G}
\DeclareMathOperator{\Hg}{Hg}
\DeclareMathOperator{\Sp}{Sp}
\DeclareMathOperator{\QQ}{\mathbb{Q}}
\DeclareMathOperator{\SL}{SL}

\DeclareMathOperator{\cone}{cone}
\DeclareMathOperator{\conv}{conv}
\DeclareMathOperator{\im}{im}
\DeclareMathOperator{\C}{\mathcal{C}}

\DeclareMathOperator{\Strong}{Strong}
\DeclareMathOperator{\transpose}{{}^t \!}  
\DeclareMathOperator{\Mat}{Mat} 
 
\DeclareMathOperator{\NN}{\mathbb{N}}

\DeclareMathOperator{\can}{can}
\DeclareMathOperator{\CC}{\mathbb{C}}
\DeclareMathOperator{\OO}{O}
\DeclareMathOperator{\RR}{\mathbb{R}}

\DeclareMathOperator{\Hom}{Hom}

\begin{document}

\title{Cycles positifs dans les vari\'et\'es ab\'eliennes}
\date{}
\author{Max Rempel}
\maketitle

\selectlanguage{french}
\begin{abstract}
 Dans la premi\`ere partie, on \'etudie
la structure de la $\RR$-alg\`ebre engendr\'ee par les classes de Hodge sur une puissance $A^e$ d'une vari\'et\'e ab\'elienne principalement polaris\'ee tr\`es g\'en\'erale $A$ de dimension $n$.
La deuxi\`eme partie est consacr\'ee \`a la comparaison de diverses notions de positivit\'e pour des cycles de codimension sup\'erieure dans $A^e$.  
En particulier, on montre qu'il y a, en toute codimension $2 \le k \le en-2$, des
classes num\'eriquement effectives qui ne sont pas pseudoeffectives, ce qui g\'en\'eralise un r\'esultat de Debarre, Ein, Lazarsfeld et Voisin.
\end{abstract}

\selectlanguage{english}
\begin{abstract}
 In the first part, we study the structure of the $\RR$-algebra generated by the Hodge classes on the self-product $A^e$ of a very general
principally polarized abelian variety $A$. In the second part, we compare various notions of positivity for cycles of higher codimension in $A^e$. In particular,
we prove that, in every codimension $2 \le k \le en-2$, there exist classes that are numerically effective but not pseudoeffective, which generalises a result
of Debarre, Ein, Lazarsfeld and Voisin.
\end{abstract}
\selectlanguage{french}

\tableofcontents

\section{Introduction}
Soit  $B$ une vari\'et\'e ab\'elienne complexe de dimension $m$. Notons
$\N^k(B)$ le $\RR$-espace vectoriel des classes de cohomologie de
cycles de codimension $k$ \`a coefficients r\'eels sur $B$ et posons
\[N^\bullet(B)=\bigoplus_{k \ge 0} \N^k(B).\]
Le produit d'intersection munit ce $\RR$-espace vectoriel d'une structure de $\RR$-alg\`ebre.

Soit $A$ une vari\'et\'e ab\'elienne principalement polaris\'ee tr\`es g\'en\'erale de dimension fix\'ee.
Le but de cet article est d'\'etudier la structure de l'alg\`ebre $\N^\bullet(A^e)$ et de divers c\^ones
de classes \og positives \fg \hskip 1mm dans $\N^k(A^e)$.

\begin{question}
\label{question_structure}
 Que peut-on dire de la structure de la $\RR$-alg\`ebre $\N^\bullet(A^e)$?
\end{question}

Il est connu que l'alg\`ebre $\N^\bullet(A^e)$ est engendr\'ee
par $\N^1(A^e)$ (cf. \cite{Tankeev} ou th\'eor\`eme \ref{Tankeev}). En utilisant l'action naturelle
de $\GL_e(\RR)$ sur $\N^\bullet(A^e)$, on montre
le th\'eor\`eme suivant (th\'eor\`eme \ref{thm_Thompson} et corollaire \ref{corol_rel_explicites}).

\begin{thm}
Soit $A$ une vari\'et\'e ab\'elienne principalement polaris\'ee tr\`es g\'en\'erale de dimension $n$.
L'id\'eal $I$ tel que 
\[\N^\bullet(A^e)=\Sym^\bullet \N^1(A^e)/I,\]
i.e., l'id\'eal des relations dans $\Sym^\bullet \N^1(A^e)$,  
est engendr\'e par des classes de cycles de codimension $n+1$ et l'on peut en d\'ecrire des g\'en\'erateurs explicitement.
En particulier, l'application
\[\Sym^k \N^1(A^e) \to \N^k(A^e)\]
est un isomorphisme si et seulement si $k \in \{0, \dots, n\}$.
\end{thm}

La deuxi\`eme partie de l'article est consacr\'ee \`a la question suivante, qui a \'et\'e abord\'ee dans \cite{DELV} pour $e=2$.
\begin{question}
\label{question_pos}
 Que peut-on dire des classes \og positives \fg \hskip 1mm dans $\N^k(A^e)$?
\end{question}

Pour pr\'eciser cette question, on rappelle d'abord diverses notions de positivit\'e qui donnent chacune lieu \`a un c\^one convexe dans $\N^k(B)$
(cf. \cite[\S 1]{DELV} et \cite[Ch. III]{Demailly}). 

\begin{enumerate}
 \item Le c\^one des classes \emph{pseudoeffectives} $\Psef^k(B)$ est le c\^one convexe ferm\'e engendr\'e par les classes de cycles effectifs.
  \item Le c\^one des classes \emph{num\'eriquement effectives (nef)}  $\Nef^k(B)$ est d\'efini comme le dual du c\^one $\Psef^{m-k}(B)$ par rapport au produit d'intersection.
\end{enumerate}

Si l'on \'ecrit $B=V/\Lambda$ avec $V$ un $\CC$-espace vectoriel et $\Lambda$ un r\'eseau dans $V$, on peut identifier une classe $\alpha \in \N^k(B)$ avec une 
$(k,k)$-forme r\'eelle sur $V$ qui s'identifie encore avec une forme hermitienne sur $\bigwedge^k V$ (cf. section \ref{section_prel}). 
Cela nous permet de d\'efinir deux autres notions de positivit\'e.

\begin{enumerate}
 \item Une classe $\alpha \in \N^k(B)$ est dite 
\emph{fortement positive} si la $(k,k)$-forme associ\'ee s'\'ecrit comme combinaison lin\'eaire convexe de
formes 
\[  i l_1 \wedge  \bar l_1 \wedge  \dots \wedge i l_k \wedge  \bar l_k\]
avec $l_j \in V^*$ pour $j=1, \dots, k$.
On obtient ainsi le c\^one ferm\'e $\Strong^k(B)$ engendr\'e par les classes fortement positives. 
\item Enfin, on dit qu'une classe est \emph{semipositive} si la forme hermitienne associ\'ee est semipositive et l'on
note $\Semi^k(B)$ le c\^one engendr\'e par ces classes.
\end{enumerate}
Le lien entre ces c\^ones est donn\'e par la cha\^ine d'inclusions \cite[Lemma 1.5]{DELV}
\begin{equation}
 \label{inclusions1}
\Sym^k \Psef^1 (B) \subset \Psef^k(B) \subset \Strong^k(B) \subset \Semi^k(B) \subset \Nef^k(B),
\end{equation}
o\`u l'on note $\Sym^k \Psef^1(B)$ le c\^one convexe engendr\'e par les produits de $k$ \'el\'ements
de $\Psef^1(B)$.
Pour $k=1, m-1$, tous les c\^ones de la cha\^ine (\ref{inclusions1}) co\"incident, de sorte que l'on a un seul c\^one qui a
\'et\'e d\'etermin\'e par Prendergast-Smith dans \cite{PS}. Pour $2 \le k \le m-2$, on se demande quelles inclusions 
sont strictes et quelles inclusions sont des \'egalit\'es.

La conjecture suivante est un cas particulier d'une conjecture de Harvey et Knapp \cite{HK}, qui a \'et\'e pr\'ecis\'ee par Lawson dans
\cite{Lawson}.
\begin{conjecture}
\label{question12}
Soit $B$ une vari\'et\'e ab\'elienne de dimension $m$. Alors on a, pour tout $1 \le k \le m$,
 \[\Psef^k(B)=\Strong^k(B).\]
\end{conjecture}

En g\'en\'erale, il est difficile de d\'ecrire les c\^ones $\Psef^k(B)$ et $\Strong^k(B)$ explicitement,
ce qui rend la v\'erification de la conjecture \ref{question12} difficile. 

Pour
$A$ une vari\'et\'e ab\'elienne principalement polaris\'ee tr\`es g\'en\'erale de dimension $n \ge 2$, on a par la proposition 5.2 de \cite{DELV},
\[\Sym^2 \Psef^1(A \times A)=\Semi^2(A \times A),\]
et
\[\Sym^{2n-2}\Psef^1(A\times A)= \Strong^{2n-2}(A \times A),\]
ce qui v\'erifie la conjecture \ref{question12} dans ces deux cas.
\begin{question}
\label{question2}
Si $A$ est une vari\'et\'e ab\'elienne principalement polaris\'ee tr\`es
g\'en\'erale de dimension $n$, a-t-on 
\[\Sym^k \Psef^1(A^e)=\Psef^k(A^e)=\Strong^k(A^e)=\Semi^k(A^e)\]
pour tout $k \in \{0, \dots, ne\}$ et tout $e \ge 2$?
\end{question}

On obtient le r\'esultat partiel suivant (th\'eor\`eme \ref{thm_inclusion_stricte_1} et proposition \ref{prop_semi4}).
\begin{thm}
\label{thm_intro2}
Soit $A$ une vari\'et\'e ab\'elienne principalement polaris\'ee tr\`es g\'en\'erale de dimension $n$ et soit $e \ge 2$. 
 \begin{enumerate}
  \item 
On a
\[\Sym^k \Psef^1(A^e) \varsubsetneq \Semi^k(A^e)\]
pour $3 \le k \le n$, et
les rayons extr\'emaux du c\^one $\Sym^k \Psef^1(A \times A)$ sont aussi extr\'emaux dans 
le c\^one  $\Semi^k(A \times A)$ pour $2 \le k \le n$.
\item Pour $n=3$, on a 
\[\Sym^4 \Psef^1(A \times A)= \Semi^4(A \times A).\]
\end{enumerate}
\end{thm}

La conjecture \ref{question12}, resp. la question de savoir si le c\^one $\Psef^k(A^e)$ co\"incide avec l'un des deux c\^ones $\Sym^k \Psef^1(A^e)$ ou
$\Semi^k(A^e)$, reste ouverte pour $k \ge 3$ et $e \ge 2$ (resp. $k \ge 2$ et $e \ge 3$).

Pour la relation entre les c\^ones $\Psef^k(A \times A)$ est $\Nef^k(A \times A)$, les auteurs de \cite{DELV} montrent que, pour $A$ une surface, on a 
\[\Psef^2(A \times A) \varsubsetneq \Nef^2(A \times A),\]
ce qui nous m\`ene \`a la question suivante.

\begin{question}
\label{question3}
Si $A$ est une vari\'et\'e ab\'elienne principalement polaris\'ee tr\`es
g\'en\'erale de dimension $n$, a-t-on 
\[\ \Psef^k(A^e) \varsubsetneq \Nef^k(A^e) \]
pour tout $2 \le k \le en-2$?
\end{question}

Concernant la question \ref{question3}, on a le r\'esultat suivant (th\'eor\`eme \ref{thm_inclusion_stricte_2}).
\begin{thm}
\label{thm_intro3}
Soit $A$ une vari\'et\'e ab\'elienne principalement polaris\'ee tr\`es g\'en\'erale de dimension $n$ et soit $e \ge 2$. 
On a
\[\Psef^k(A^e) \varsubsetneq \Nef^k (A^e)\]
pour $2 \le k \le en-2$. 
\end{thm}

Les points clefs dans les d\'emonstrations du th\'eoreme \ref{thm_intro2} et du th\'eor\`eme \ref{thm_intro3} 
sont d'un c\^ot\'e l'utilisation du fait que tous les c\^ones dans la cha\^ine (\ref{inclusions1}) sont invariants sous
l'action de $\GL_e(\RR)$ sur $\N^\bullet(A^e)$ et d'autre c\^ot\'e la caract\'erisation des c\^ones $\Sym^k \Psef^1(A^e)$
et $\Semi^k(A^e)$. Expliquons rapidement les r\'esultats que l'on obtient concernant ces deux c\^ones: 
par un r\'esultat de Prendergast-Smith \cite{PS}, le c\^one $\Psef^1(A)$ est homog\`ene sous l'action de $\GL_e(\RR)$ et ses classes extr\'emales forment
une orbite, ce qui permet une description facile des g\'en\'erateurs de $\Sym^k \Psef^1(A^e)$ (proposition \ref{prop_ext_rays}). 
D'autre c\^ot\'e, le c\^one des classes semipositives correspond naturellement 
\`a un c\^one de matrices (en regardant les matrices repr\'esentant les formes hermitiennes associ\'ees
aux classes dans $\N^k(A^e)$), de sorte que l'on pourrait esp\'erer calculer ces matrices et d'en obtenir des in\'equations d\'efinissant $\Semi^k(A^e)$. 
Pour $k=2$ et $e=2$ les matrices sont tr\`es petites, de sorte que les calculs ne posent pas de probl\`emes alors que,
pour $k \ge 3$ et $e \ge 2$ (resp. pour $k=2$ et $e \ge 3$), les matrices deviennent rapidement beaucoup plus grandes, ce qui rend les calculs tr\`es p\'enibles d\'ej\`a pour
$k=3$ et $e=2$. On r\'esoud ce probl\`eme en d\'ecrivant la structure de ces matrices plus conceptuellement:
\'ecrivons $A=U/\Gamma$ avec $U$ un $\CC$-espace vectoriel et $\Gamma$ un r\'eseau dans $U$.
L'action de $\GL_e(\RR)$ sur $\N^\bullet(A^e)$ provient d'une action sur $U^{\oplus e}$, et en utilisant la th\'eorie des repr\'esentations, on montre que les matrices repr\'esentant
les formes hermitiennes associ\'ees aux classes dans $\N^k(A^e)$ sont des matrices diagonales par blocs, et 
ces blocs correspondent aux modules irr\'eductibles dans une d\'ecomposition du $\GL_e(\RR)$-module $\bigwedge^k U^{\oplus e}$.
Pour deux modules diff\'erents isomorphes dans une d\'ecomposition de $\bigwedge^k U^{\oplus e}$, les blocs correspondants sont
identiques et un bloc se calcule \`a partir des matrices repr\'esentant
l'action de $\GL_e(\RR)$ sur le module irr\'eductible de $\bigwedge^k U^{\oplus e}$ correspondant
(proposition \ref{prop_bilinear_decomp} et remarque \ref{rem_algorithme}).
 Pour calculer les matrices repr\'esentant les formes hermitiennes
associ\'ees aux classes dans $\N^k(A^e)$, on est ainsi essentiellement ramen\'e \`a
\begin{enumerate}
 \item d\'ecomposer $\bigwedge^k U^{\oplus e}$ en $\GL_e(\RR)$-modules irr\'eductibles,
 \item calculer les matrices repr\'esentant l'action de $\GL_e(\RR)$ sur les $\GL_e(\RR)$-modules irr\'eductibles apparaissant dans $\bigwedge^k U^{\oplus e}$.
\end{enumerate}

\vspace{5mm}

Je tiens \`a remercier chaleureusement O. Debarre pour avoir sugg\'er\'e ce sujet et pour les nombreux conseils
qu'il m'a donn\'es.

\section{La structure de l'alg\`ebre $\N^\bullet(A^e)$}

\subsection{Le groupe de Hodge d'une vari\'et\'e ab\'elienne}
\label{section_prel}

Soit $B$ une vari\'et\'e ab\'elienne complexe de dimension $m$ et
\'ecrivons $B=V/\Lambda$, o\`u $V$ est un $\CC$-espace vectoriel de dimension $m$ et
$\Lambda \subset V$ est un r\'eseau. On a $\Lambda=H_1(B, \ZZ)$ et
$H^1(B, \ZZ)=\Lambda^*$, o\`u l'on note $\Lambda^*$ le r\'eseau dual \`a $\Lambda$. Soit $K=\QQ$ ou $\RR$, et posons
$V_K=\Lambda \otimes_{\ZZ} K$.
Pour $k \ge 1$, le cup-produit fournit un isomorphisme
\begin{equation*}
H^k(B, \ZZ) \simeq  \bigwedge^k \Lambda^*. 
\end{equation*}

Rappelons que l'on a (cf. \cite[Thm. 1.4.1]{BL})
\begin{equation}
 \label{kkforme}
H^{p,q}(B)=\bigwedge^p V^* \otimes \bigwedge^q \overline{V}^*,
\end{equation}
o\`u $V^*=\Hom_{\CC}(V,\CC)$ et 
$\overline{V}^*$ est l'espace vectoriel des formes $\CC$-antilin\'eaires sur $V$.

Soit
\[N^k(B)_{\QQ}=H^{k,k}(B) \cap H^{2k}(B, \QQ)\]
le $\QQ$-espace vectoriel des classes de Hodge (rationnelles) de degr\'e $2k$ sur $B$. 
Posons $\N^k(B)=N^k(B)_{\QQ} \otimes \RR$ et notons
$N^\bullet(B)=\bigoplus_{k \ge 0} \N^k(B)$
la $\RR$-alg\`ebre d\'efinie dans l'introduction. 
Remarquons que l'on a par  (\ref{kkforme}) une injection
de $\N^k(B)$ dans l'espace vectoriel r\'eel des $(k,k)$-formes r\'eelles sur
$V$, que l'on note $\bigwedge_{\RR}^{(k,k)} V^*$.
Cet espace vectoriel est isomorphe \`a celui des
formes hermitiennes sur $\bigwedge^k V$, que l'on note $\mathcal{H}$: si l'on
munit $V$ de coordonn\'ees  $(z_1, \dots, z_n)$ et $V^*$ des coordonn\'ees duales,
cet isomorphisme  est donn\'e par 
\begin{align*}
\label{bij_formes}
 \mathcal{H}   & \to \bigwedge^{(k,k)}_{\RR} V^*\\
  \sum_{I,J} h_{IJ} dz_I \otimes d \bar z_J   & \mapsto   \sqrt{-1}^{k^2} \sum_{I,J} h_{IJ} dz_I \wedge d \bar z_J, \nonumber
\end{align*}
o\`u $I=\{i_1, \dots, i_k\}\subset\{1, \dots n\},  dz_I= dz_{i_1} \wedge \dots \wedge dz_{i_k}$, 
et de m\^eme pour $J$. 
On peut ainsi associer \`a une classe de Hodge sur $B$ de degr\'e $2k$ une forme hermitienne
sur $\bigwedge^k V$, que l'on notera $H_\alpha$.

Soit $J: V_{\mathbb{R}} \to V_{\mathbb{R}}$ la structure complexe associ\'ee \`a
\[H_1(B, \mathbb{C})=H^{-1,0}(B) \oplus H^{0,-1}(B).\]
On obtient un morphisme
\begin{eqnarray*}
 h_J: \mathbb{S}^1 & \to & \SL(V_{\mathbb{R}}) \\
       e^{\sqrt{-1}\cdot \theta} & \mapsto & \cos \theta \cdot id_V + \sin \theta \cdot J,
\end{eqnarray*}
qui agit  par multiplication par $z$ sur $H^{1,0}(B)$ et  par multiplication par
$\bar z$ sur $H^{0,1}(B)$ (cf. \cite[Prop. 17.1.1, Rem. 17.1.2]{BL}).

\begin{definition}
Le \emph{groupe de Hodge} de $B$, not\'e $\Hg(B)$, est le plus petit sous-groupe de $\SL(V_{\mathbb{R}})$ d\'efini sur $\mathbb{Q}$
contenant l'image de $h_J$.
\end{definition}

Soit $\theta \in H^2(B, \ZZ)$ une polarisation de $B$.
Par l'isomorphisme $H^2(B, \QQ)=\bigwedge^2 V_{\QQ}^*$,
$\theta$ d\'efinit une forme altern\'ee non-d\'eg\'en\'er\'ee 
\[\omega_\theta: V_{\QQ} \times V_{\QQ} \to \QQ.\]
On note $\Sp(V_{\QQ}, \omega_\theta)$ le groupe symplectique associ\'e \`a $\omega_\theta$, qui est naturellement un
sous-groupe alg\'ebrique de $\SL(V_{\QQ})$. Le groupe de Hodge $\Hg(B)$ est un sous-groupe de $\Sp(V_{\QQ}, \omega_\theta)$
pour toute polarisation $\theta$  \cite[Prop. 17.3.2]{BL}. En particulier, $\Hg(B)$ agit naturellement sur $V_{\QQ}$. 

\begin{nota}
Si $G$ est un groupe et $M$ est un $G$-module, on note 
$M^G$ l'ensemble des invariants dans $M$ sous l'action de $G$.
\end{nota}

Regardons $H^1(B, \QQ)=V_{\QQ}^*$ comme repr\'esentation duale de $\Hg(B)$.  Par l'isomorphisme
$H^k(B, \QQ) =\bigwedge^k V_{\QQ}^*$, on obtient ainsi une structure de $\Hg(B)$-module sur $H^k(B, \QQ)$
pour $k \ge 1$.
Par l'\'egalit\'e
\cite[17.3.3]{BL}
\begin{equation*}
\N^k(B)_{\QQ}=H^{2k}(B, \QQ)^{\Hg(B)},
\end{equation*}
on est ainsi ramen\'e \`a un calcul des invariants pour d\'eterminer les classes de Hodge.

\subsection{Classes de Hodge sur les puissances d'une vari\'et\'e ab\'elienne}
\label{section_struc_puiss}

Soit $A$ une vari\'et\'e ab\'elienne (pas forcement principalement polaris\'ee tr\`es g\'en\'erale).
On veut \'etudier l'alg\`ebre des classes de Hodge sur des puissances de $A$.
Ecrivons $A=U/\Gamma$, o\`u
$U$ est un $\CC$-espace vectoriel et $\Gamma$ est un r\'eseau dans $U$. 
 Si l'on \'ecrit
$A^e=(U \otimes_{\ZZ} W_{\ZZ})/(\Gamma \otimes_{\ZZ} W_{\ZZ})$, on obtient une injection
\begin{align*}
\End(W_{\ZZ})  & \to \End(A^e)\\
g & \mapsto u_g,
\end{align*} 
o\`u $u_g$ est d\'efini par l'application lin\'eaire
\begin{align*}
 U \otimes_{\ZZ} W_{\ZZ} & \to  U \otimes_{\ZZ} W_{\ZZ} \\
 u \otimes w & \mapsto  u\otimes \transpose gw.
\end{align*}
Posons, pour $g \in \End(W_{\ZZ})$ et pour $\alpha \in H^\bullet(A^e, \QQ)$,
\[  g \cdot \alpha=u_g^*\alpha.\]
Cela induit une action de $\GL(W_{\QQ})$ sur 
$H^\bullet(A^e, \QQ)$, et l'on a 
\[H^{k}(A^e, \QQ) \simeq \bigwedge^k U^* \otimes W_{\QQ} \] 
en tant que $\GL(W_{\QQ})$-modules, o\`u $\GL(W_{\QQ})$ agit sur $W_{\QQ}$ tautologiquement.  

Or, par un r\'esultat de Hazama \cite[Cor. 1.11]{Hazama},
on a 
\begin{equation}
\label{hg_product}
\Hg(A^e) = \Hg(A),
\end{equation} 
et $\Hg(A)$ agit diagonalement sur $H^1(A^e, \QQ)$, i.e.,
\begin{equation*}
\forall g \in \Hg(A) \, \, \forall u \in U_{\QQ}^* \, \, \forall w \in W_{\QQ} \hspace{5mm} g \cdot (u \otimes w)=g \cdot u \otimes w.
\end{equation*}
On obtient ainsi une structure de $\Hg(A) \times \GL(W_{\QQ})$-module sur
$H^k(A^e, {\QQ})$. Par construction,
$\N^\bullet(A^e)_{\QQ}$ est invariant sous l'action de $\GL(W_{\QQ})$.

\subsection{G\'en\'erateurs et relations pour les classes de Hodge}
\label{section_genrel}

Soit maintenant $(A, \theta)$ une vari\'et\'e ab\'elienne principalement polaris\'ee tr\`es g\'en\'erale et repr\'enons
les notations de la section pr\'ec\'edente.
Par un r\'esultat de Tankeev, l'alg\`ebre $\N^\bullet(A^e)$ est engendr\'ee par des classes de codimension $1$ (th\'eor\`eme \ref{Tankeev}),
de sorte que l'application canonique $\Sym^\bullet \N^1(A^e) \to \N^\bullet(A^e)$ est surjective. On d\'etermine 
 des g\'en\'erateurs de $\N^\bullet(A^e)$ et l'id\'eal des relations dans $\Sym^\bullet \N^1(A^e)$ (proposition \ref{thm_gen} et corollaire \ref{corol_rel_explicites}).
Le groupe de Hodge d'une vari\'et\'e ab\'elienne principalement polaris\'ee tr\`es g\'en\'erale \'etant le groupe symplectique $\Sp(U_{\QQ})$
(par rapport \`a la polarisation principale), 
on est ramen\'e \`a l'\'etude 
de l'alg\`ebre des invariants $(\bigwedge^\bullet U^*_{\QQ} \otimes W_{\QQ})^{\Sp(U_{\QQ})}$ en tant que $\GL(W_{\QQ})$-module  pour obtenir ces r\'esultats. 
Comme cette alg\`ebre
a \'et\'e \'etudi\'ee par Thompson dans \cite{Thompson}, notre travail consiste essentiellement \`a traduire ses r\'esultats dans notre cadre.

Soit $K= \QQ$ ou $\RR$. On a
\[H^k (A^e, K)  \simeq \bigwedge^k (U_K^* \otimes_{\ZZ} W_{\ZZ}).\]

\begin{thm}[Tankeev]
\label{Tankeev}
Soit $A$ une vari\'et\'e ab\'elienne principalement polaris\'ee tr\`es g\'en\'erale. Alors
on a 
\begin{equation}
\label{hdg_general}
\Hg(A^e)=\Sp(U_{\QQ})
\end{equation}
 et $\N^\bullet(A^e)_{\QQ}$ est engendr\'e par des classes de codimension $1$.
\end{thm}

\begin{proof}
Pour montrer (\ref{hdg_general}), il suffit  par (\ref{hg_product}) de remarquer que l'on a $\Hg(A)=\Sp(U_{\QQ})$ pour
$A$ tr\`es g\'en\'erale (cf. \cite[Prop. 17.4.2]{BL}).
Il s'ensuit 
\[\N^\bullet(A^e)_{\QQ}=\left( \bigwedge^{2k} V_{\QQ}^* \right )^{\Sp(U_{\QQ})}.\]
Par un th\'eor\`eme de Howe \cite[Thm. 2]{howe} (cf. aussi \cite[pp. 529-530]{Ribet}),
l'alg\`ebre \`a droite est engendr\'ee par des classes de degr\'e $2$,
ce qui fournit le r\'esultat souhait\'e.
\end{proof}

L'action de
$\Sp(U_{\QQ}) \times \GL(W_{\QQ})$ sur $V_{\QQ}^*=U_{\QQ}^* \otimes W_{\QQ}$ s'\'etend naturellement \`a 
une action de $\Sp(U_{\RR}) \times \GL(W_{\RR})$ sur $V_{\RR}^*=U_{\RR}^* \otimes W_{\RR}$, et
comme $\Sp(U_{\QQ})$ est dense dans $\Sp(U_{\RR})$, on a
\begin{equation*}
\N^\bullet(A^e)=\left( \bigwedge^{2k} V_{\QQ}^* \right )^{\Sp(U_{\QQ})} \hskip -9mm \otimes_{\QQ} \RR=  \left( \bigwedge^{2k} V_{\RR}^* \right )^{\Sp(U_{\RR})}.
\end{equation*}

Comme expliqu\'e plus haut,  nous traduisons maintenant les r\'esultats de Thompson \cite{Thompson} dans notre cadre.
Pour faciliter les notations pour la suite, on pose $W=W_{\RR}$. La proposition suivante correspond \`a
 \cite[proposition 2.2]{Thompson}.

\begin{prop}
\label{corol_decomp_N}
On a un isomorphisme de $\GL(W)$-modules
\[\N^\bullet(A^e) \simeq \bigoplus_\sigma \mathbb{S}_{\sigma}(W),\]
o\`u l'on note $\mathbb{S}_\sigma$ le foncteur de Schur correspondant au tableau de Young $\sigma$ et o\`u
l'on prend la somme sur les $\sigma$ tels que 
\begin{enumerate}
\item chaque ligne de $\sigma$ a un nombre pair d'\'el\'ements, 
\item la premi\`ere ligne a au plus $2n$ \'el\'ements et
\item le nombre de lignes est au plus $e$.
\end{enumerate}
\end{prop}

\begin{corol}
\label{corol_sym2}
 On a 
\[\N^1(A^e) \simeq \Sym^2 W\]
en tant que $\GL(W)$-modules.
\end{corol}

\begin{proof}
Le tableau de Young ayant une ligne \`a $2$ cases correspond au $\GL(W)$-module $\Sym^2 W$, et l'on a donc
$ \N^1(A^e)=\Sym^2 W$.
\end{proof}

De plus, on a le r\'esultat suivant \cite[2.3.8]{Weyman}.
\begin{prop}
\label{young_decomp}
On a un isomorphisme de $\GL(W)$-modules
\begin{equation*}
\Sym^\bullet ( \N^1(A^e))=\bigoplus_{\sigma} \mathbb{S}_\sigma(W),
\end{equation*}
o\`u l'on note $\mathbb{S}_\sigma$ le foncteur de Schur correspondant au tableau de Young $\sigma$ et o\`u
l'on prend la somme sur les $\sigma$ tels que 
\begin{enumerate}
\item chaque ligne de $\sigma$ a un nombre pair d'\'el\'ements, 
\item le nombre de lignes est au plus $e$.
\end{enumerate}
\end{prop}

Comme $\Sym^{2k}W$ correspond au tableau de Young ayant une ligne \`a $2k$ cases, on
voit qu'il existe une injection canonique
\begin{equation}
\label{rem_plong_sym} 
\varepsilon_{k}:  \Sym^{2k}W \to \Sym^k (\N^1(A^e)) .
\end{equation}
Par un r\'esultat d'Abeasis \cite[Thm. 3.1]{Abeasis}, on en d\'eduit le r\'esultat suivant.
\begin{prop}
\label{prop_abeasis}
Soit $k$ un entier positif fix\'e, et soit $I_k$ l'id\'eal engendr\'e par $\varepsilon_k(\Sym^{2k} W)$ dans $\Sym^\bullet(\N^1(A^e))$. Alors
$I_k$ est $\GL(W)$-invariant et
 \[I_k = \bigoplus_\sigma \mathbb{S}_{\sigma}(W),\]
o\`u l'on note $\mathbb{S}_\sigma$ le foncteur de Schur correspondant au tableau de Young $\sigma$ et o\`u
l'on prend la somme sur les $\sigma$ tels que 
\begin{enumerate}
\item chaque ligne de $\sigma$ a un nombre pair d'\'el\'ements, 
\item la premi\`ere ligne a au moins $2k$ \'el\'ements et
\item le nombre des lignes est au plus $e$.
\end{enumerate}
\end{prop}

Le r\'esultat principal de Thompson se traduit alors comme suit dans notre cadre \cite[Thm. 2.3]{Thompson}.
\begin{thm}
\label{thm_Thompson}
Soit $A$ une vari\'et\'e ab\'elienne principalement polaris\'ee tr\`es g\'en\'erale de dimension $n$, et
soit $I_{n+1}$ l'id\'eal engendr\'e par $\varepsilon_{n+1}(\Sym^{2n+2} W)$ dans $\Sym^\bullet(\N^1(A^e))$.
Alors 
\[\Sym^\bullet(\N^1(A^e))/I_{n+1}    = \N^\bullet(A^e).\]
En particulier, l'application
\[\Sym^k \N^1(A^e) \to \N^k(A^e)\]
est un isomorphisme si et seulement si
$k \in \{0, \dots,n\}$.
\end{thm}

\begin{corol}
\label{corol_injection}
Pour toute classe $\alpha \in \N^1(A^e)$ non nulle et pour tout 
$0 \le k \le n-1$, l'application
\begin{eqnarray*}
 \N^k(A^e) & \to & \N^{k+1}(A^e) \\
\beta &\mapsto& \alpha \cdot \beta
\end{eqnarray*}
est injective.
\end{corol}

D\'eterminons maintenant des g\'en\'erateurs de $\N^\bullet(A^e)$.
Comme $A$ est principalement polaris\'ee, on peut identifier $A$ avec sa vari\'et\'e ab\'elienne duale $\widehat A$,
ce qui nous permet de voir le fibr\'e de Poincar\'e $\mathcal{P}$ comme un
fibr\'e sur $A \times A$.
Soient $p_{i_1, \dots, i_l}:A^e \to A^l$ les projections.
Posons
\[\theta_i=p_i^*\theta \qquad \text{et} \qquad \lambda_{jk}=c_1(p_{jk}^*\mathcal{P}).\]
\begin{rem}
\label{rem_kuenneth}
Notons $A_i$ le $i$-\`eme facteur de $A^e$. 
Dans la d\'ecomposition de K\"unneth de $H^2(A^e, \RR)$, la classe $\theta_i$ est contenue dans la composante correspondant
 \`a
$H^2(A_i, \RR)$, et $\lambda_{jk}$ appartient \`a la composante correspondant \`a $H^1(A_j , \RR) \otimes H^1(A_k, \RR)$ \cite[Lemma 14.1.9]{BL}.
\end{rem}
On a la g\'en\'eralisation suivante de la proposition 3.1 de \cite{DELV}.
\begin{prop}
\label{thm_gen}
Soit $(A, \theta)$ une vari\'et\'e ab\'elienne principalement polaris\'ee tr\`es g\'en\'erale.
La $\RR$-alg\`ebre
$\N^\bullet(A^e)$ est engendr\'ee par les $\theta_i, 1\le i \le e$ et les $\lambda_{jk}, 1\le j<k \le e$.
\end{prop}

\begin{proof}
Par le corollaire \ref{corol_sym2}, on a $\N^1(A^e)\simeq \Sym^2 W$ et donc
\[\dim(\N^1(A^e))=\binom{e+1}{2}.\]
 Comme les $\theta_i, \lambda_{jk}, 1 \le i \le e, 1 \le j< k \le e$ forment une famille
libre de dimension $\binom{e}{2}+e=\binom{e+1}{2}$ dans $\N^1(A^e)$ par la remarque \ref{rem_kuenneth}, on obtient le r\'esultat souhait\'e.
\end{proof}

D\'eterminons maintenant comment $\GL(W)$ agit sur la base 
\[\{\theta_i, \lambda_{jk} \mid 1 \le i \le e, 1 \le j< k \le e\}\]
de $\N^1(A^e)$. Soit
\begin{equation}
\label{eqn_base}
v_{sr}=u_s \otimes w_r \hspace{5mm} 1 \le s \le n, 1 \le r \le e
\end{equation}
une base
de $V^*=U^* \otimes_{\RR} W$ telle que l'on a dans les coordonn\'ees associ\'ees (cf. \cite[Lemme 3.6.4]{BL})
\begin{eqnarray}
\label{coord_classes1}
\theta_i & = & \sum_{s=1}^n \sqrt{-1} dz_{si} \wedge d\bar z_{si}, \\
\lambda_{jk} & = & \sum_{s=1}^n  \sqrt{-1}dz_{sj}  \wedge d\bar z_{sk}  + \sqrt{-1}dz_{sk}  \wedge d\bar z_{sj} \notag
\end{eqnarray}
pour $1\le i \le e, 1\le j<k \le e$.

\begin{prop}
\label{prop_iso_sym_explicite}
Soit $g \in \GL(W)$ et soit $\rho(g)$ la matrice repr\'esentant $g$ dans la base $\{w_1, \dots, w_e\}$.
Sur l'ensemble des g\'en\'erateurs $\{\theta_i, \lambda_{jk} \mid 1 \le i \le e ,1 \le j < k \le e\}$
de $\N^\bullet(A^e)$, la matrice $\rho(g)$ agit par:
\begin{align*}
g\theta_i &=   \sum_{j=1}^e \rho(g)_{ji}^2 \theta_j + \sum_{1 \le j < k \le e} \rho(g)_{ji}\rho(g)_{ki} \lambda_{jk}, \\
g \lambda_{jk} &= 2\sum_{i=1}^e \rho(g)_{ij}\rho(g)_{ik} \theta_l + \sum_{1 \le u < v \le e}  (\rho(g)_{vj}\rho(g)_{uk}+\rho(g)_{uj}\rho(g)_{vk}) \lambda_{uv}.  
\end{align*}
\end{prop}

\begin{proof}
Cela r\'esulte d'un calcul direct dans les coordonn\'ees (\ref{coord_classes1}) ou
des arguments de   \cite[\S 3]{DELV}.
\end{proof}

\begin{rem}
\label{rem_identifi}
Si $\{w_1, \dots, w_e\}$ est une base de $W$, il d\'ecoule de la proposition \ref{prop_iso_sym_explicite}
que l'on obtient l'isomorphisme $\N^1(A^e)\simeq \Sym^2W$ en identifiant 
$\theta_i$ avec $w_i^2$ et $\lambda_{jk}$ avec $2w_jw_k$ pour $1\le i \le e, 1\le <j<k \le e$.
\end{rem}

Soit $m: A^e  \to A$ l'application d'addition.
 Alors on a 
\[m^*\theta=\sum_{i=1}^e \theta_i + \sum_{1 \le j < k \le e} \lambda_{jk}\] 
et donc
\begin{equation}
\label{eqn_rel1}
\left(\sum_{i=1}^e \theta_i + \sum_{1 \le j < k \le e} \lambda_{jk} \right)^{n+1}=0
\end{equation}
dans $H^{2n+2}(A^e, \RR)$.
Regardons  maintenant $ \left(\sum_{i=1}^e \theta_i + \sum_{1 \le j<k \le e} \lambda_{jk} \right)^{n+1}$ comme polyn\^ome dans $\Sym^{n+1} \N^1(A^e)$ et
\'ecrivons
\[\left(\sum_{i=1}^e \theta_i + \sum_{1 \le j<k \le e} \lambda_{jk} \right)^{n+1}=\sum_l P_l,\]
o\`u chaque polyn\^ome $P_l$ d\'efinit une classe dans un facteur de la d\'ecomposition de K\"unneth de $H^{2n+2}(A^e,\RR)$ 
(cf. remarque \ref{rem_kuenneth} et exemple \ref{ex_dim2}). On obtient 
ainsi pr\'ecis\'ement un polyn\^ome $P_l \not =0$ pour chaque composante de K\"unneth, et par la relation (\ref{eqn_rel1}), chaque polyn\^ome
$P_l$ doit repr\'esenter la classe $0$ dans $H^{2n+2}(A^e, \RR)$. On appelle
les relations $P_l=0$ dans $H^{2n+2}(A^e, \RR)$ ainsi obtenues, les \emph{relations induites par la relation} (\ref{eqn_rel1}).

\begin{corol}
\label{corol_rel_explicites}
Soit $A$ une vari\'et\'e ab\'elienne principalement polaris\'ee tr\`es g\'en\'erale de dimension $n$.
Soit $I$ l'id\'eal des relations dans $\Sym^\bullet \N^1(A^e)$, de sorte que
 $\Sym^\bullet \N^1(A^e)/I=\N^\bullet(A^e)$. Alors $I$ est engendr\'e
par les relations induites par la relation
\begin{equation}
\label{eqnkunneth}
\left(\sum_{i=1}^e \theta_i + \sum_{1 \le j<k \le e} \lambda_{jk} \right)^{n+1}=0
\end{equation}
dans  $H^{2n+2}(A^e, \RR)$ via la d\'ecomposition de K\"unneth.
\end{corol}

\begin{proof}
On a 
\[\binom{2n+1+e}{2n+2}=\dim \Sym^{2n+2}W\]
composantes dans la d\'ecomposition de K\"unneth de $H^{2n+2}(A^e, \RR)$, de sorte que l'\'equation (\ref{eqnkunneth})  induit 
$\binom{2n+1+e}{2n+2}$ relations qui forment une famille libre dans 
\\ $\Sym^{n+1}(\N^1(A^e))$. Par
la proposition \ref{prop_abeasis} et le th\'eor\`eme \ref{thm_Thompson}, on a
$I \cap \Sym^{n+1}(\N^1(A^e)) \simeq \Sym^{2n+2}W$ , d'o\`u l'on d\'eduit  le r\'esultat souhait\'e
 par une comparaison des dimensions.
\end{proof} 

\begin{ex}
\label{ex_dim2}
Supposons $n=2$. On d\'etermine les relations dans $\Sym^3 \N^1(A \times A)$ selon le corollaire \ref{corol_rel_explicites}. On a 
\[(\theta_1+\theta_2+\lambda)^3=\theta_1^3 + \theta_2^3 +\lambda^3+3\theta_1^2\theta_2 + 3\theta_1\theta_2^2 +3\theta_1\lambda^2 +3\theta_2\lambda^2+6\theta_1\theta_2\lambda\]
et
\begin{alignat*}{2}
 \theta_1^3 & \in H^6(A) \otimes H^0(A),  \qquad & \theta_2^3  &\in H^0(A) \otimes H^6(A),  \\
 3\theta_1^2\theta_2+3\theta_1\lambda^2   & \in H^4(A) \otimes H^2(A),  & 3\theta_1\theta_2^2+3\theta_2\lambda^2  &\in H^2(A) \otimes H^4(A), \\
 3\theta_1^2 \lambda  & \in H^5(A)\otimes H^1(A),  & 3\theta_2^2 \lambda   &\in H^1(A) \otimes H^5(A), \\
6\theta_1\theta_2\lambda+\lambda^3 & \in H^3(A) \otimes H^3(A). & & 
\end{alignat*}
Si l'on note $I$ l'id\'eal des relations dans la $\RR$-alg\`ebre $\N^\bullet(A \times A)$, on a donc
\[I= \left \langle \theta_1^3,   \theta_1^2\theta_2+\theta_1\lambda^2, \theta_1^2 \lambda ,6\theta_1\theta_2\lambda+\lambda^3 , \theta_2^3 , \theta_1\theta_2^2+\theta_2\lambda^2,\theta_2^2 \right \rangle.\]
\end{ex}

Finissons par un r\'esultat qui nous permet de faire des calculs d'intersections dans $\N^\bullet(A^e)$.

\begin{prop}
\label{intersections}
 Soit $A$ une vari\'et\'e ab\'elienne principalement polaris\'ee tr\`es g\'en\'erale de dimension $n$. 
Les seuls mon\^omes en $\theta_1, \theta_2$ et $\lambda$ non nuls de degr\'e $2n$ sont
de la forme $\theta_1^{n-k}\theta_2^{n-k} \lambda^{2k}$ pour $k \in \{0, \dots, n\}$, et l'on a
\[\theta_1^{n-k}\theta_2^{n-k} \lambda^{2k}=(-1)^{k}(2k)!(n-k)!^2 \binom{n}{k}.\]
\end{prop}

\begin{proof}
Il est clair que les m\^onomes de degr\'e $2n$ de la forme $\theta_1^{n-k}\theta_2^{n-k} \lambda^{k}$ sont pr\'ecis\'ement les mon\^omes
d\'efinissant une classe dans
$H^n(A, \RR) \otimes H^n(A, \RR)$  par rapport \`a la d\'ecomposition de K\"unneth (cf. remarque \ref{rem_kuenneth}). Evidemment, c'est la seule composante non nulle,
ce qui montre la premi\`ere partie de l'\'enonc\'e.
Pour la deuxi\`eme partie, on suppose que $\theta_1, \theta_2$ et $\lambda$ sont donn\'es comme dans (\ref{coord_classes1}).
Alors $\theta_1^{n-k}\theta_2^{n-k}$ correspond \`a la $(2(n-k),2(n-k))$-forme
\[\sum_{ 1 \le i_1 < \dots <i_{n-k}\le n \atop n+1 \le j_1 < \dots <j_{n-k}\le 2n }
  (n-k)!^2 z_{i_1} \wedge i \bar z_{i_1} \wedge \dots \wedge z_{i_{n-k}} \wedge i \bar z_{i_{n-k}} \wedge z_{j_1} \wedge i \bar z_{j_1} \wedge \dots \wedge z_{j_{n-k}} \wedge i \bar z_{j_{n-k}}.\]
Si l'on \'ecrit 
\[\lambda^{2k}=\sum_{I,J} a_{IJ} dz_{i_1}\wedge d\bar z_{j_1} \wedge \dots \wedge dz_{i_{2k}}\wedge d\bar z_{j_{2k}} \]
 avec $a_{IJ} \in \RR$ et $I,J$ des multi-indices de longeur $2k$ dans $\{1, \dots, n\}$, 
on v\'erifie facilement que l'on a $a_{IJ} \not =0$ et $a_{IJ} dz_{i_1}\wedge d\bar z_{j_1} \wedge \dots \wedge dz_{i_{2k}}\wedge d\bar z_{j_{2k}}\wedge \theta_1^{n-k}\theta_2^{n-k} \not =0$ si et seulement
si $a_{IJ} dz_{i_1}\wedge d\bar z_{j_1} \wedge \dots \wedge dz_{i_{2k}}\wedge d\bar z_{j_{2k}}$ est de la forme
\begin{align*}
(2k)!z_{i_1} & \wedge i\bar z_{i_{1+n}}  \wedge z_{i_{1+n}} \wedge i \bar z_{i_1} \wedge  \dots  \wedge z_{i_k} \wedge i\bar z_{i_{k+n}} \wedge z_{i_{k+n}} \wedge i \bar z_{i_k} \\
& = (-1)^{3k} (2k)!z_{i_1} \wedge i\bar z_{i_1} \wedge z_{i_{1+n}} \wedge i \bar z_{i_{1+n}} \wedge  \dots  \wedge z_{i_k} \wedge i\bar z_{i_k} \wedge z_{i_{k+n}} \wedge i \bar z_{i_{k+n}}.
\end{align*}
Or il y a $\binom{n}{k}$ telles formes $a_{IJ} dz_{i_1}\wedge d\bar z_{j_1} \wedge \dots \wedge dz_{i_{2k}}\wedge d\bar z_{j_{2k}}$, et pour chaque telle forme on a 
\[a_{IJ} dz_{i_1}\wedge d\bar z_{j_1} \wedge \dots \wedge dz_{i_k}\wedge d\bar z_{j_k} \wedge \theta_1^{n-k}\theta_2^{n-k} =(-1)^{k}(2k)!(n-k)!^2 ,\]
de sorte que l'on obtient 
\[\theta_1^{n-k}\theta_2^{n-k} \lambda^{2k}=(-1)^{k}(2k)!(n-k)!^2 \binom{n}{k}.\]
\end{proof}

\begin{rem}
Si $A$ est de dimension $n$, on a
\[\mu^n=\sum_{k=0}^n 4^{n-k} (-1)^k \theta_1^{n-k} \theta_2^{n-k} \lambda^{2k}=\sum_{k=0}^n 4^{n-k}(2k)!(n-k)!^2 \binom{n}{k}>0.\]
\end{rem}

\begin{ex}
 Soit $n=2$. Alors on a (cf. \cite[\S 4]{DELV})
\[ \theta_1^2 \theta_2^2=4, \qquad \theta_1 \theta_2 \lambda^2=-4, \qquad \lambda^4=24, \qquad \mu^2=96.\]
\end{ex}

\section{Classes positives}
\label{section_classes_positives}

Soit $A$ toujours une vari\'et\'e ab\'elienne principalement polaris\'ee tr\`es g\'en\'erale de dimension $n$. 
Dans ce chapitre, on \'etudie la cha\^ine d'inclusions
\begin{equation}
\label{chain_cones}
\Sym^k \Psef^1 (A^e) \subset \Psef^k(A^e) \subset (\Strong^k(A^e) \subset )\Semi^k(A^e) \subset \Nef^k(A^e).
\end{equation}
On commence avec des r\'esultats sur le c\^one $\Sym^k \Psef^1(A^e)$ (section \ref{section_conesk})
et ensuite on regarde le c\^one des classes semipositives $\Semi^k(A^e)$ (section \ref{section_semipos}). 
Les r\'esultats de ces deux sections nous permettent alors de d\'eduire les th\'eor\`emes
\ref{thm_intro2} et \ref{thm_intro3} de 
l'introduction 
(th\'eor\`eme \ref{thm_inclusion_stricte_1} et th\'eor\`eme \ref{thm_inclusion_stricte_2}) dans les sections
\ref{section_classes_comp1} et \ref{section_classes_nefs}. 

\begin{rem}
Pour la vari\'et\'e ab\'elienne $A^e$, tous les c\^ones dans la cha\^ine (\ref{chain_cones}) sont 
invariants sous l'action de $\GL(W)$ sur $\N^\bullet(A^e)$ (cf. \S \ref{section_prel} et \cite[Prop. 1.6 et p. 12]{DELV}).
\end{rem}

\subsection{Le c\^one $\Sym^k \Psef^1(A^e)$ }
\label{section_conesk}

Dans cette section,
on montre que le c\^one $\Sym^k \Psef^1(A^e)$ engendr\'e par les intersections de $k$ diviseurs pseudoeffectifs 
est ferm\'e et l'on d\'etermine ses rayons extr\'emaux 
pour $1 \le k \le n$ et $e\ge 2$ (proposition \ref{prop_ext_rays}).

Regardons d'abord le cas $k=1$ (cf. \cite[\S 5.2]{BL} et \cite{PS}).
 Posons $\End_{\RR}(A^e)=\End(A^e) \otimes_{\ZZ} \RR$. Notons $\End^s_{\RR}(A^e)$ l'espace des \'el\'ements 
sym\'etriques par rapport \`a l'involution de Rosati de $\End_{\RR}(A^e)$ et $T_x$ la translation par $x$ dans $A^e$ pour un $x \in A^e$ fix\'e.
Ayant $\widehat{A^e} \simeq A^e$, on obtient un isomorphisme
\begin{eqnarray*}
 \N^1(A^e) & \simeq  &\End^s_{\RR}(A^e)\\
 L & \mapsto &  \phi_L
\end{eqnarray*}
de $\RR$-espaces vectoriels,
o\`u $\phi_L(x)=T_x^*L\otimes L^{-1}$. Si l'on identifie $\End_{\RR}(A^e)$ avec l'alg\`ebre de matrices $\Mat_e(\RR)$, l'involution de Rosati 
correspond \`a l'involution qui envoie une matrice sur sa transpos\'ee, de sorte que $\N^1(A^e)$ est isomorphe \`a l'espace des matrices $e \times e$ r\'eelles sym\'etriques $\Symm_e(\RR)$. 
Comme les valeurs propres de $\phi_L$ sont les m\^emes que ceux d'une matrice repr\'esentant la forme hermitienne $H_L$ associ\'ee \`a $L$ \cite[Lemme 2.4.5]{BL}, le c\^one 
 $\Nef^1(A^e)$ correspond au c\^one  $\Symm_e^+(\RR)$ des matrices r\'eelles sym\'etriques positives.
De plus, l'action de $\GL(W)$ sur $\N^1(A^e)$ correspond \`a l'action de $\GL_e(\RR)$ sur $\Symm_e(\RR)$  donn\'ee par \cite[Theorem 4.2]{PS}
\[\forall g \in \GL_e(\RR) \, \, \forall M \in \Symm_e(\RR) \qquad g \cdot M= g M \transpose g,\]  
de sorte que le c\^one $\Nef^1(A^e)$ est homog\`ene sous l'action de $\GL(W)$.

Ayant  $\phi_{\theta_i}(x_1, \dots, x_e)=(0, \dots, 0, x_i, 0, \dots, 0)$, on voit que
$\theta_i$ correspond \`a une matrice sym\'etrique de rang $1$ dans $\Symm_e^+(\RR)$, d'o\`u l'on d\'eduit que $\theta_i$ est 
extr\'emal dans $\Nef^1(A^e)$. Toutes les matrices sym\'etriques positives de rang $1$ \'etant congruentes
sous l'action de $\GL(W)$, on obtient comme  cas particulier de \cite[Thm 4.3]{PS}  l'\'enonc\'e suivant.

\begin{prop}
\label{corol_ext_1}
Soit $A$ une vari\'et\'e ab\'elienne principalement polaris\'ee tr\`es g\'en\'erale. Le c\^one $\Psef^1(A^e)=\Nef^1(A^e)$ est homog\`ene sous l'action
de $\GL(W)$ et l'ensemble de ses rayons extr\'emaux correspond \`a l'orbite $\GL(W) \cdot \theta_1$.
\end{prop}

Avant d'\'etudier le c\^one $\Sym^k \Psef^1(A^e)$ pour $k \ge 2$, on aura besoin des notations suivantes.

\begin{nota}
Si $E$ est un sous-ensemble de $\RR^d$, on note
$\conv(E)$ l'enveloppe convexe de  $E$ et $\cone(E)$ le c\^one convexe engendr\'e par $E$.
Pour un c\^one convexe $\C \subset \RR^d$, on note
$\extr(\C)$ la r\'eunion des rayons extr\'emaux de $\mathcal{C}$.
\end{nota}

\begin{rem}
\label{rem_SO}
Soit $\mathcal{B}:=\{w_1, \dots,  w_e\}$ une base  de $W$. Munissons $W$ du produit int\'erieur tel que la base $\mathcal{B}$ est orthonorm\'ee et notons 
 $\OO(W)$ le groupe orthogonal par rapport \`a ce produit int\'erieur. Dans les coordonn\'ees associ\'ees \`a
$\mathcal{B}$,  $\GL(W)$ s'identifie avec $\GL_e(\RR)$, et si l'on identifie $\Sym^2 W$ avec $\N^1(A^e)$ comme d\'ecrit dans la remarque \ref{rem_identifi}, l'action de $g \in \GL(W)$ sur $\theta_1$ ne d\'epend que de la premi\`ere colonne de la matrice repr\'esentant  $g$, de sorte
que $GL(W)\cdot \theta_1=\RR_+^*(\OO(W) \cdot \theta_1)$.
\end{rem}

\begin{nota}
Posons pour la suite
\[E_k=\{ g_1\theta_1 \cdots g_k\theta_1 \mid g_1, \dots, g_k \in \GL(W)\} \subset \N^k(A^e).\]
Pour $\alpha$ et $\beta \in \N^\bullet(A^e)$, on \'ecrit $\alpha \sim \beta$, si 
$\alpha$ et $\beta$ sont dans la m\^eme orbite sous l'action de $\GL(W)$.
\end{nota}

\begin{prop}
\label{prop_ext_rays}
Soit $A$ une vari\'et\'e ab\'elienne principalement polaris\'ee tr\`es g\'en\'erale de dimension $n$.
Le c\^one $\Sym^k \Psef^1(A^e)$ est ferm\'e et l'on a l'inclusion
\[\ext (\Sym^k \Psef^1(A^e)) \subset E_k \]
qui est une \'egalit\'e pour $k \in \{0, \dots, n\}$.
\end{prop}

\begin{proof}
Posons $\G=\GL(W)$.
Comme un c\^one ferm\'e qui ne contient pas de droite est l'enveloppe convexe de ses rayons extr\'emaux, toute classe $\alpha \in \Psef^1(A^e)$ 
s'\'ecrit comme combinaison convexe des classes dans l'orbite $G \cdot \theta_1$ par la proposition \ref{corol_ext_1}.
On a donc
\[\Sym^k \Psef^1(A^e)=\cone \{g_1\theta_1 \cdots g_k \theta_1 \mid g_1, \dots, g_k \in \G \}.\]
En particulier, $\Sym^k \Psef^1(A \times A)$ est le c\^one engendr\'e par l'image de l'application continue
(cf. remarque \ref{rem_SO})
\begin{eqnarray*}
 \varphi: \OO(W) \times \cdots \times \OO(W) & \to & \N^k(A^e) \\
       (g_1, \dots, g_k) & \mapsto & g_1 \theta_1 \cdots g_k\theta_1,
\end{eqnarray*}
i.e., $\Sym^k \Psef^1(A^e)$ est engendr\'e par un ensemble compact et est donc ferm\'e.
Or, l'enveloppe convexe de l'image de $\varphi$ est compacte de sorte que tout $x \in \conv(\im(\varphi))$ 
s'\'ecrit comme combinaison convexe d'un nombre fini d'\'el\'ements de $\im(\varphi)$.
Comme, par la remarque \ref{rem_SO},
\[\Sym^k \Psef^1(A^e)=\cone(\conv(\im(\varphi)))=\{\lambda x \mid \lambda \in \RR^+, x \in \conv(\im(\varphi))\}, \]
on a 
\[\ext (\Sym^k \Psef^1(A^e) ) \subset E_k.\]
Montrons maintenant 
\[E_k  \subset  \ext (\Sym^k \Psef^1(A^e) )\]
pour $1 \le k \le n$. On raisonne par r\'ecurrence sur $k$.
Pour \'eviter des indices, on le d\'emontre pour
$e=2$, le raisonnement pour le cas $e \ge 3$ \'etant similaire.

Supposons donc $e=2$.
Pour $k=1$, c'est la proposition \ref{corol_ext_1}. 
Soit maintenant $k \ge 2$ et soit $g_1 \theta_1 \cdots g_k \theta_1\in E_k$. Comme
   \[g_1\theta_1 \cdots g_k\theta_1 \sim \theta_1 \cdot g_1^{-1}g_2 \theta_1  \cdots g_1^{-1}g_k \theta_1=:\alpha,\]
et comme  $\ext(\Sym^k \Psef^1(A \times A))$ est 
invariant sous l'action de $G$, il suffit de montrer que
$\alpha$ est extr\'emal.
Ecrivons $\alpha=\sum_j s_j$ 
avec
\[s_j=  g_1^j \theta_1 \cdots g_k^j \theta_1\]
et  
\[g_i^j \theta_1=(a_i^j)^2 \theta_1 +(b_i^j)^2 \theta_2 +a_i^j b_i^j \lambda.\]
Supposons qu'il existe un $s_j$ tel que 
$b_i^j \not = 0$ pour tout $i \in \{1, \dots, k\}$. Cela implique que $s_j$ contient $\theta_2^k$ comme terme avec un coefficient strictement positif, ce qui n'est pas 
possible, car les coefficients de ce terme sont toujours $\ge 0$ et $\alpha$ ne contient pas de tel terme.
On peut donc supposer, pour tout $j$, 
$s_j=\theta_1 \cdot   g_2^j \theta_1 \cdots g_k^j \theta_1$, i.e.,
\[\theta_1 \cdot \left ( g_1^{-1}g_2 \theta_1 \cdots  g_1^{-1}g_k \theta_1 \right )
                  = \sum_j \theta_1 \cdot   (g_2^j \theta_1 \cdots g_k^j \theta_1).\]
Par le corollaire \ref{corol_injection},  on a, pour $1 \le k \le n$,
\[ g_1^{-1}g_2 \theta_1  \cdots  g_1^{-1}g_k \theta_1 = \sum_j g_2^j \theta_1 \cdots g_k^j \theta_1, \]
ce qui fournit le r\'esultat par r\'ecurrence sur $k$: tous les $g_2^j \theta_1 \cdots g_k^j \theta_1$ doivent \^etre
proportionnels \`a $g_1^{-1}g_2 \theta_1  \cdots  g_1^{-1}g_k \theta_1$, donc tous les $s_j$ sont proportionnels \`a $\alpha$.
\end{proof}

Finissons par la proposition suivante qui nous servira dans la section \ref{section_classes_nefs}.
\begin{prop}
\label{prop_im_inv}
 Soit $g \in \GL(W)$, soit $1 \le k \le n-1$ et soit
\begin{eqnarray*}
 \varphi_g: \N^k(A^e) & \to &  \N^{k+1}(A^e) \\
             \alpha & \mapsto & g\theta_1 \cdot \alpha.
\end{eqnarray*}
Alors $\varphi_g^{-1}(\Sym^{k+1}\Psef^1(A^e))=\Sym^k \Psef^1(A^e)$.
\end{prop}

\begin{proof}
Il faut montrer que $g\theta_1 \cdot \alpha \in \Sym^{k+1}\Psef^1(A^e)$ implique que
$\alpha$ est contenu dans $\Sym^k \Psef^1(A^e)$ (l'autre sens est clair). Par la stabilit\'e de $\Sym^k \Psef^1(A^e)$
sous l'action de $\GL(W)$, il suffit de le montrer pour $\varphi_{\id}$. Soit 
$\theta_1 \cdot \alpha \in \Sym^{k+1}\Psef^1(A^e)$. Par la proposition pr\'ec\'edente, on peut alors \'ecrire
\[\theta_1 \cdot \alpha=\sum_j g_1^j \theta_1 \cdots g_k^j \theta_1,\]
o\`u la somme \`a droite contient un nombre fini de termes.
Par l'argument utilis\'e dans la preuve de la proposition \ref{prop_ext_rays}, on obtient
\[\theta_1 \cdot \alpha=\theta_1 \cdot \sum_j g_2^j \theta_1 \cdots g_k^j \theta_1.\]
Par le corollaire \ref{corol_injection}, on a 
donc 
\[\alpha= \sum_j g_2^j \theta_1 \cdots g_k^j \theta_1 \in \Sym^k \Psef^1(A^e).\]
\end{proof}

\subsection{Le c\^one semipositif}
\label{section_semipos}

Soit $A=U/\Gamma$ toujours une vari\'et\'e ab\'elienne principalement polaris\'ee tr\`es g\'en\'erale de dimension $n$, 
o\`u $U$ est un espace vectoriel complexe et $\Gamma$
un r\'eseau dans $U$.  Soit 
$W$ un espace vectoriel r\'eel de dimension $e$ et \'ecrivons $V^* \simeq U^* \otimes_{\RR} W$. On note 
$V_{\RR}$ l'espace vectoriel r\'eel engendr\'e par les vecteurs $v_{sr}, 1 \le s \le n, 1 \le r \le e$ (cf. (\ref{eqn_base})). 
 
Dans ce chapitre, on \'etudie les c\^ones semipositifs $\Semi^k(A^e)$. Dans une premi\`ere \'etape, on remarque qu'au lieu de regarder
la positivit\'e des formes hermitiennes $H_\alpha$ sur $\bigwedge^k V$ associ\'ees aux classes $\alpha \in \N^k(A^e)$, on peut \'etudier la positivit\'e d'une forme
bilin\'eaire sym\'etrique $B_\alpha$ sur $\bigwedge^k V_{\RR}$.
 Pour $k$ fix\'e, on montre qu'il existe une base telle que les matrices repr\'esentant 
les formes sym\'etriques $B_\alpha$ se d\'ecomposent en blocs correspondant aux facteurs irr\'eductibles de $\bigwedge^k V_{\RR}$ en tant que  $\GL(W)$-module  (proposition \ref{prop_bilinear_decomp}). La d\'emonstration de la proposition \ref{prop_bilinear_decomp} 
fournira  aussi une m\'ethode pour calculer les matrices  $b_\alpha$ repr\'esentant les formes sym\'etriques $B_\alpha$
\`a partir des repr\'esentations irr\'eductibles de $\GL(W)$ (remarque \ref{rem_algorithme}), ce qui nous permet de
d\'eduire que le c\^one $\Semi^k(A^e)$ 
ne d\'epend pas de la dimension de $A$ pour $n \ge k$ (corollaire \ref{corol_ind_dim}).
Enfin, on montre que, pour $1 \le k \le n$, on a des isomorphismes 
\begin{align*}
\N^k(A \times A) & \simeq \Symm_{k+1}(\RR), \\
\N^{2n-k}(A \times A) & \simeq \Symm_{k+1}(\RR),
\end{align*}
qui associent \`a une classe semipositive une matrice sym\'etrique positive (proposition \ref{prop_decomp}).

On reprend les notations de la section  \ref{section_genrel}. La base $\{w_1, \dots, w_e\}$ de $W$  est dite la 
\emph{base standard}
de $W$ et
\begin{equation*}
\{v_{s_1r_1} \wedge \dots \wedge v_{s_kr_k} \mid 1 \le s_i \le n, 1 \le r_i \le e, i \in \{1, \dots, k\} \}
\end{equation*}
est dite la \emph{base standard} de $\bigwedge^k V$ (cf. (\ref{eqn_base})). Le produit int\'erieur sur $\bigwedge^k V$ telle que
la base standard est une base orthonorm\'ee est dit \emph{produit int\'erieur standard}.

\subsubsection{D\'ecomposition des formes hermitiennes}

Afin d'associer \`a une classe dans $\N^k(A \times A)$ une forme bilin\'eaire sym\'etrique sur un espace vectoriel r\'eel,
et pour  \'etudier l'ensemble des formes ainsi obtenues, on a d'abord besoin d'un lemme technique.

Soit $\rho_k: \End(W) \to \End(\bigwedge^k V)$ la repr\'esentation induite par la repr\'esentation tautologique $W$ et notons
$\rho_k^{\mathcal B}(g)$ la matrice repr\'esentant $\rho_k(g)$ par rapport aux coordonn\'ees standards sur $W$ resp. sur $\bigwedge^k V$.  Rappelons que l'on a un isomorphisme 
\begin{align*}
\End(W)  & \to \End(A)\otimes \RR  \\
g & \mapsto  u_g
\end{align*}
et que, par construction, la repr\'esentation analytique de $u_g$ est donn\'ee par $\rho_1$.
Ayant de plus un isomorpisme $\phi: \N^1(A^e) \to \End^s(A^e) \otimes \RR$, on peut identifier tout
$\alpha \in \N^1(A^e)$ avec un  $g_\alpha \in \End(W)$.
Par  \cite[Lemme 2.4.5]{BL}, on a
\begin{equation}
\label{eqn_eq_matrices}
\rho_1^{\mathcal{B}}(g_\alpha)=h_\alpha^{\mathcal{B}},
\end{equation}
o\`u l'on note  $h_\alpha^{\mathcal{B}}$ la matrice repr\'esentant la forme hermitienne $H_\alpha$ associ\'ee \`a $\alpha \in \N^1(A^e)$
par rapport aux coordonn\'ees standards sur $W$ resp. sur $V$.

\begin{lemme}
\label{equality_maps}
Dans les coordonn\'ees standards de $W$, resp. de $\bigwedge^k V$, on a 
\begin{equation*}
\rho^{\mathcal{B}}_k(g_\alpha)=\frac{1}{k!}h^{\mathcal{B}}_{\alpha^k}
\end{equation*}
pour tout $\alpha \in \N^1(A^e)$.
\end{lemme}

\begin{proof}
Rappelons que le produit ext\'erieur d'une forme hermitienne sur $V$ est d\'efinie par
\[(\bigwedge^k H)(v_1 \wedge \dots \wedge v_k, w_1 \wedge \dots \wedge w_k)=\det( H(v_i, w_j)_{1 \le i,j \le k})\]
pour $v_1 \wedge \dots \wedge v_k$ et $ w_1 \wedge \dots \wedge w_k$ dans $\bigwedge^k V$.
Comme le produit d'intersection de deux classes dans $\N^\bullet(A^e)$ correspond au produit ext\'erieur des $(k,k)$-formes repr\'esentant les deux classes,
 on a $(\bigwedge^k H_\alpha)=\frac{1}{k!} H_{\alpha^k}$ et donc
\[\frac{1}{k!}h_{\alpha^k}^{\mathcal{B}}=\det(( h_\alpha^{\mathcal{B}})_{1 \le i,j \le k}).\]
En m\^eme temps, $\rho_1(g_\alpha)$ induit naturellement un endomorphisme $\bigwedge^k \rho_1(g_\alpha)$ sur $\bigwedge^k V$ en posant
$\bigwedge^k \rho_1(g_\alpha) (v_1 \wedge \dots \wedge v_k)=  \rho_1(g_\alpha) v_1 \wedge \dots \wedge \rho_1(g_\alpha) v_k$. Dans la base standard de
$\bigwedge^k V$, on a 
\[\bigwedge^k \rho_1^{\mathcal{B}}(g_\alpha)=\det( \rho_1(g^{\mathcal{B}}_\alpha)_{1 \le i,j \le k}).\]
Ayant $\rho_k^{\mathcal{B}}(g_\alpha)=\bigwedge^k \rho_1^{\mathcal{B}}(g_\alpha)$, on obtient le r\'esultat souhait\'e par (\ref{eqn_eq_matrices}).
\end{proof}

Remarquons que toutes les matrices $h_\alpha^{\mathcal{B}}$ sont r\'eelles, de sorte que  les matrices $h_{\alpha^k}^{\mathcal{B}}$
sont \'egalement r\'eelles. Comme les $\alpha^k$ engendrent $\N^k(A^e)$, on en d\'eduit que, pour tout $\beta \in \N^k(A^e)$, 
la matrice $h_\beta^{\mathcal{B}}$ est r\'eelle.  Cela nous permet d'identifier ces matrices avec une forme bilin\'eaire sym\'etrique sur $\bigwedge^k V_{\RR}$. On obtient ainsi une application
\begin{align*}
B: \N^k(A^e) & \to \Sym^2(\bigwedge^k V_{\RR}^*) \\
\alpha &\mapsto B_\alpha,
\end{align*}
qui associe \`a une classe semipositive une forme bilin\'eaire sym\'etrique semipositive sur $\bigwedge^k V_{\RR}$. 

La proposition suivante montre qu'il existe une base de $\bigwedge^k V_{\RR}$ telle que, pour tout $\alpha \in \N^k(A^e)$, 
la matrice $b_\alpha$ repr\'esentant la forme bilin\'eaire $B_\alpha$ se d\'ecompose en blocs correspondant aux $\GL(W)$-modules irr\'eductibles de $\bigwedge^k V_{\RR}$.

\begin{prop}
\label{prop_bilinear_decomp}
Il existe une d\'ecomposition
\[\bigwedge^k V_{\RR}=\bigoplus_l M_l\]
en facteurs irr\'eductibles  telle que  l'application
$B: \N^k(A^e)  \to \Sym^2(\bigwedge^k V_{\RR}^*)$ se factorise par $\bigoplus_l \Sym^2(M_l^*)$.
\end{prop}

\begin{proof}
Comme la base standard de $\bigwedge^k V_{\RR}$ est orthogonale par rapport au produit int\'erieur standard,
il existe par l'astuce unitaire une transformation orthogonale 
$a: \bigwedge^k V_{\RR} \to \bigwedge^k V_{\RR}$ telle que
la nouvelle base est adapt\'ee \`a la d\'ecomposition $\bigwedge^k V_{\RR} =\bigoplus_l M_l$. Par la proposition \ref{equality_maps}, on a donc
\[\transpose a \frac{1}{k!} b^{\mathcal{B}}_{\beta^k} a= \transpose a  \rho^{\mathcal{B}}_k(g_\beta) a=a^{-1}  \rho^{\mathcal{B}}_k(g_\beta) a\]
pour tout $\beta \in \N^1(A^e)$.
Par construction, les matrices $a^{-1}  \rho^{\mathcal{B}}_k(g_\beta) a$ se d\'ecomposent en blocs correspondant aux facteurs
irr\'eductibles de $\bigwedge^k V_{\RR}$, et comme les $\beta^k$ engendrent $\N^k(A^e)$, on en d\'eduit le r\'esultat souhait\'e.
\end{proof}

\begin{rem}
\label{rem_algorithme}
 La d\'emonstration de la proposition \ref{prop_bilinear_decomp} montre comment on peut calculer les matrices
$b_\alpha$ repr\'esentant les formes $B_\alpha$ associ\'ees aux classes $\alpha$ dans $\N^k(A^e)$:
soit 
\begin{equation}
\label{ref_decomp}
 \bigwedge^k V_{\RR}=\bigoplus_l M_l
\end{equation}
 une d\'ecomposition en facteurs irr\'eductibles orthogonaux deux \`a deux par rapport au produit int\'erieur standard sur $\bigwedge^k V_{\RR}$.  Choisissons maintenant une base
 orthonorm\'ee de $\bigwedge^k V_{\RR}$ adopt\'ee \`a cette d\'ecomposition. Si l'on note $\rho_{M_l}$ la repr\'esentation $\End(W) \to\End(M_l)$, on a
$\rho_k(g)=\bigoplus_l \rho_{M_l}(g)$ et par la proposition \ref{equality_maps}, on a
$\rho_{M_l}(g_\beta)=\frac{1}{k!} b_{\beta^k}|_{M_l}$ pour tout $\beta \in \N^1(A^e)$ dans les coordonn\'ees choisies. Cela nous permet
de calculer les matrices $b_\alpha$ pour $\alpha \in \N^k(A^e)$  comme suit 
(on l'explique en d\'etail pour $e=2$, le cas 
g\'en\'eral \'etant similaire):
on calcule d'abord les repr\'esentations $\rho_{M_l}:\End(W) \to \End(M_l)$ dans une base orthonorm\'ee de $M_l$ 
par rapport au produit int\'erieur induit sur $M_l$ en tant que sous-module de $\bigwedge^k V_{\RR}$. Supposons
que l'on obtient 
\begin{eqnarray*}
 \rho_{M_l}: \Mat_2(\RR) & \to & \Mat_{\dim (M_l)}(\RR)  \\
        \begin{pmatrix} a & b \\ c & d \end{pmatrix}  & \mapsto &  (p_{ij}(a,b,c,d))_{ij},
\end{eqnarray*}    
o\`u les $p_{ij}$ sont des polyn\^omes homog\`enes de degr\'e $k$.
\'Ecrivons de plus
\[p_{ij}(a,b,b,d)=\sum_{|l|=k} t_{ij}^{(l)}  a^{l_1}d^{l_2}b^{l_3}\]
avec $l=(l_1, l_2, l_3)$ et $t_{ij}^{(l)} \in \RR$. On a d'autre part
\[(a\theta_1 + d\theta_2 +b\lambda)^k=\sum_{|l|=k} \binom{k}{l_1, l_2, l_3} a^{l_1}d^{l_2}b^{l_3} \theta_1^{l_1} \theta_2^{l_2} \lambda^{l_3 }.\]
Si $\alpha \in \N^k(A^e)$ est une classe repr\'esent\'ee par le polyn\^ome
\[\sum_{|l|=k} x_{l_1, l_2, l_3} \theta_1^{l_1}\theta_2^{l_2}\lambda^{l_3},\]
la matrice $b_\alpha|_{M_l}$ repr\'esentant $B_\alpha|_{M_l}$, est alors donn\'ee par
\begin{equation}
 \label{eqn_polynome_classe}
\frac{1}{k!} b_\alpha|_{M_l} = \left( \sum_{|l|=k} t_{ij}^{(l)}\binom{k}{l_1, l_2, l_3}^{-1} x_{l_1, l_2, l_3} \right )_{ij}.
\end{equation}

 Remarquons enfin que si $M_l$ et $M_l'$ sont
deux modules isomorphes dans la d\'ecomposition (\ref{ref_decomp}), il suffit de calculer $b_\alpha|_{M_l}$ pour \'etudier la
positivit\'e de $\alpha$. Cela d\'ecoule du fait que
$b_\alpha|_{M_l'}$ doit \^etre une matrice semblable \`a $b_\alpha|_{M_l}$, de sorte que les deux matrices ont les 
m\^emes valeurs propres. 
\end{rem}

\begin{corol}
\label{corol_ind_dim}
 Les c\^ones $\Semi^k(A \times A)$ ne d\'ependent pas de la dimension de $A$ pour $n \ge k$.
\end{corol}

\begin{proof}
 Par la d\'emonstration de la proposition \ref{prop_bilinear_decomp} et la remarque \ref{rem_algorithme}, il suffit de montrer qu'un module irr\'eductible appara\^it dans
la d\'ecomposition de $\bigwedge^k W^{\oplus n}$ si et seulement s'il appara\^it dans la d\'ecomposition 
de $\bigwedge^k W^{\oplus k}$ pour $n \ge k$: notons $l$ un multi-indice $(l_1, \dots, l_k)$ et posons
$|l|=\sum_{i=1}^k l_i$. On a\[\bigwedge^k W^{\oplus k} = \bigoplus_{|l|=k} (\bigwedge^{l_1} W \otimes \dots \otimes \bigwedge^{l_k} W)^{\oplus m_l}\]
et
\[\bigwedge^k W^{\oplus n} = \bigoplus_{|l|=k} (\bigwedge^{l_1} W \otimes \dots \otimes \bigwedge^{l_k} W)^{\oplus \binom{n}{k} m_l},\]
o\`u $m_l$ est un entier positif, d'o\`u le r\'esultat.
\end{proof}

\subsubsection{Le cas $A \times A$}

Regardons maintenant le cas $e=2$. On montre d'abord que, pour $1 \le k \le n$, on a des isomorphismes 
$\N^k(A \times A) \to \Symm_{k+1}(\RR)$ (resp. $\N^{2n-k}(A \times A) \to \Symm_{k+1}(\RR)$) qui envoient des classes semipositives sur des matrices 
positives (proposition \ref{prop_decomp}). Expliquons l'id\'ee pour $n=k$: ayant $V_{\RR}= W^{ \oplus k}$, le $\GL(W)$-module
$\Sym^k(W)$ est un facteur irr\'eductible de $\bigwedge^k V_{\RR}$. On choisit alors une base telle que les matrices $b_\alpha$ se d\'ecomposent en blocs
correspondant aux facteurs irr\'eductibles de $\bigwedge^k V_{\RR}$. En prenant la projection sur le bloc correspondant \`a une copie de $\Sym^k W$, on obtient l'isomorphisme
souhait\'e.

Ensuite, on applique la m\'ethode d\'ecrite dans la remarque \ref{rem_algorithme} afin de calculer l'isomorphisme
$\N^k(A \times A) \to \Symm_{k+1}(\RR)$ explicitement (lemme \ref{lemme_mat_M}) et pour obtenir
 des in\'equations d\'efinissant les 
c\^ones semipositifs $\Semi^k(A\times A)$ pour $A$ de dimension $3$ et pour $2 \le k \le 4$.

Commen\c{c}ons par un lemme technique.

\begin{lemme}
\label{lemme_sym_mat}
Soit $\rho:\End(W) \to \End(\Sym^k W)$ la repr\'esentation standard et 
$\rho_{\mathcal{B}}:\Mat_2(\RR) \to \Mat_{k+1}(\RR)$ l'application induite pour des coordonn\'ees sur $W$ et $\Sym^k W$ fix\'ees. L'espace vectoriel engendr\'e
par $\rho_{\mathcal{B}}(\Symm_2(\RR))$ est de dimension $\frac{(k+1)(k+2)}{2}=\dim \Sym^2(\Sym^k W)$.
\end{lemme}

\begin{proof}
 On choisit une base $\{w_1, w_2\}$ de $W$ et la base 
$\{w_1^iw_2^{k-i} \mid 0 \le i \le k\}$ de $\Sym^k W$. Soit $D$ la matrice diagonale avec $D_{ii}=\binom{n}{i}$ pour $0 \le i \le k$.
Comme $D$ est inversible, il suffit de montrer que l'espace engendr\'e par les matrices  
$\rho_{\mathcal{B}}(\Symm_2(\RR))D$ est de dimension $\frac{(k+1)(k+2)}{2}$ . Pour cela, on proc\`ede comme suit. D'abord on identifie $\Symm_2(\RR)$ avec $\Sym^2 W$ par l'application
\[\begin{pmatrix} a & b \\ b & c \end{pmatrix}  \mapsto aw_1^2 + bw_1w_2+cw_2^2.\]
 Alors
 on a $( aw_1^2 + bw_1w_2+cw_2^2)^k \in \Sym^k(\Sym^2 W)$. Soit $\phi:\Sym^k(\Sym^2W)\simeq \Sym^2(\Sym^k W)$ l'isomorphisme donn\'e par la r\'eciprocit\'e d'Hermite (cf. par exemple \cite[Lemma 3]{Tambour}). Comme les $( aw_1^2 + bw_1w_2+cw_2^2)^k$ engendrent
$\Sym^k(\Sym^2 W)$, les $\phi(( aw_1^2 + bw_1w_2+cw_2^2)^k)$ engendrent $\Sym^2(\Sym^k W)$, de sorte
qu'il suffit de montrer que l'on a
\begin{equation}
\label{equality_phi}
\forall g=\begin{pmatrix} a & b \\ b & c \end{pmatrix} \in \Symm_2(\RR) \qquad
\phi(( aw_1^2 + bw_1w_2+cw_2^2)^k)=\rho_{\mathcal{B}}(g)D,
\end{equation} 
o\`u l'on identifie $\Sym^2(\Sym^k W)$ avec $\Symm_{k+1}(\RR)$ (voir ci-dessous pour les d\'etails).

Pour montrer l'\'egalit\'e (\ref{equality_phi}), on  rappelle d'abord que
 l'isomorphisme $\phi$ 
 est donn\'e par
\begin{align*}
\phi: \Sym^k(\Sym^2W) & \to \Sym^2(\Sym^k W) \hookrightarrow \Sym^k W \otimes \Sym^k W, \\
(w_1^2)^{r_1} (w_1 w_2)^{r_2}(w_2^2)^{r_3} & \mapsto (w_1 \otimes w_1)^{\ast r_1} \ast  (w_1 \otimes w_2 +w_2 \otimes w_1)^{\ast r_2} \ast (w_2 \otimes w_2)^{\ast r_3}, 
\end{align*}
o\`u $(u_1 \otimes u_2) \ast (v_1 \otimes v_2)=(u_1v_1 \otimes u_2v_2) \in \Sym^2 W \otimes \Sym^2 W$ pour
 $u_1, u_2, v_1, v_2 \in W$.
Un calcul montre alors que l'image de  $( aw_1^2 + bw_1w_2+cw_2^2)^k$  par $\phi$ est
\begin{equation}
\label{equation_phi2}
 \sum_{r_1+r_2+r_3=k} \binom{k}{r_1, r_2, r_3} a^{r_1}b^{r_2}c^{r_3}
\sum_{s=0}^{r_2}  \binom{r_2}{s}
 w_1^{r_1+r_2-s}w_2^{s+r_3} \cdot w_1^{r_1+s}w_2^{r_2-s+r_3}.
\end{equation}
 Calculons maintenant $\rho_{\mathcal{B}}(g)D$. Pour $0 \le j \le k$, on a
\begin{align*}
g \cdot w_1^jw_2^{k-j} & =  (aw_1+bw_2)^j(bw_1+dw_2)^{k-j} \\
	& =  \sum_{i=0}^k \left ( \sum_{l=0}^i \binom{j}{i-l}\binom{k-j}{l}a^{i-l}b^{j-i+2l}c^{k-j-l} \right ) w_1^i w_2^{k-i}, 
\end{align*}
et ayant $ \binom{j}{i-l}\binom{k-j}{l}= \binom{k}{i-l, j-i+2l, k-j-l} \binom{j-i+2l}{l}\binom{k}{j}^{-1}$, on obtient ainsi
\begin{equation}
\label{equation_diagonal}
(\rho_{\mathcal{B}}(g)D)_{ij}= \sum_{l=0}^i \binom{k}{i-l, j-i+2l, k-j-l} \binom{j-i+2l}{l}a^{i-l}b^{j-i+2l}d^{k-j-l}, 
\end{equation}
et l'on v\'erifie facilement que cette matrice est sym\'etrique. Cela nous permet de voir $\eta(g):=\rho_{\mathcal{B}}(g)D$ comme un \'el\'ement dans
$\Sym^2(\Sym^k W)$ via l'isomorphisme $\Sym^2(\Sym^k W) \simeq \Symm_{k+1}(\RR)$ qui envoie 
$\eta(g)$ sur $\sum_{0 \le i \le j \le k} \eta(g)_{ij} w_1^i w_2^{k-i} \cdot w_1^j w_2^{k-j}$ .
En posant $r_1=i-l, r_2=j-i+2l, r_3=k-j-l$ pour $i,j,l$ fix\'es, le terme associ\'e de la somme donnant $\eta(g)_{ij}$ s'\'ecrit 
\[ \binom{k}{r_1, r_2, r_2} \binom{r_2}{l}a^{r_1}b^{r_2}d^{r_3},\]
et l'on a 
\begin{align*}
i & = r_1+r_2-l,   & k-i & =l+r_3,  \\
j & =r_1+l, & k-j & =r_2-l+r_3,
\end{align*}
ce qui fournit le r\'esultat souhait\'e en comparant avec
(\ref{equation_phi2}).
\end{proof}

\'Etudions maintenant la d\'ecomposition de
$\bigwedge^k V_{\RR}$ en $\GL(W)$-modules irr\'eductibles. Comme on a $V_{\RR} \simeq W^{\oplus n}$ en tant que
$\GL(W)$-modules, on remplace $\bigwedge^k V_{\RR}$ par $\bigwedge^k W^{\oplus n}$ dans la suite. Or
\begin{align*}
\bigwedge^k W^{\oplus n}
 & =\bigoplus_{s=\max\{0, k-n\}}^{\lfloor \frac{k}{2} \rfloor}
														\left( \det(W)^{\otimes s} \otimes W^{\otimes k-2s}   \right )^{\binom{n}{k-2s} \binom{2s}{s}} \\
												& = \bigoplus_{s=\max\{0, k-n\}}^{\lfloor \frac{k}{2} \rfloor}
															\left ( \bigoplus_{i=0}^{\lfloor \frac{k-2s}{2} \rfloor} 
															\det(W)^{\otimes i +s} \otimes \Sym^{k-2s-2i} W   \right)^{\binom{n}{k-2s} \binom{2s}{s}}.
\end{align*}
Posons $M_{k,l}=\det(W)^{\otimes l}  \otimes \Sym^{k-2l} W  $ pour $0 \le l \le \lfloor \frac{k}{2} \rfloor$
et notons $m_{k,l}$ la multiplicit\'e de $M_{k,l}$ dans la
d\'ecomposition de $\bigwedge^k W^{\oplus n}$.
Soit
\[p_l: \bigoplus_l \Sym^2(M_{k,l}^*)^{\oplus m_{k,l}} \to \Sym^2(M_{k,l}^*)\]
une projection sur une des copies de $\Sym^2(M_{k,l}^*)$. 

La proposition suivante fournit alors les isomorphismes $\N^k(A \times A) \simeq \Symm_{k+1}(\RR)$
(resp. $\N^{2n-k}(A \times A) \to \Symm_{k+1}(\RR)$) qui envoient les classes semipositives sur des matrices
semipositives pour $1 \le k \le n$ (apr\`es un choix de coordonn\'ees sur $\Sym^k W^*$ (resp. sur $\Sym^k W^* \otimes \det(W^*)^{\otimes n-k}$)).

\begin{prop}
\label{prop_decomp}
Soit $A$ une vari\'et\'e ab\'elienne principalement polaris\'ee tr\`es g\'en\'erale de dimension $n$ et soit $0 \le k \le 2n$. Alors les $\GL(W)$-morphismes
\[p_l \circ B:\N^k(A \times A) \to \Sym^2(M_{k,l}^*)\]
sont surjectifs pour tout $l$. En particulier, pour $1 \le k \le n$,
les applications 
\[p_0 \circ B: \N^k(A \times A) \to \Sym^2(\Sym^k W^*)\]
resp.
\[p_{n-k} \circ B: \N^{2n-k}(A \times A) \to  \Sym^2(\Sym^k W^* \otimes \det(W^*)^{\otimes n-k})\]
sont des isomorphismes de $\GL(W)$-modules
qui envoient toute classe semipositive sur une forme bilin\'eaire sym\'etrique semipositive. 
\end{prop}

\begin{proof}
 Rappelons que l'on a un isomorphisme $\N^1(A^e) \simeq \Symm_2(\RR)$, et que l'on note $g_\beta$ la matrice sym\'etrique ainsi associ\'ee \`a $\beta \in \N^1(A^e)$. 
Soit $\rho_{M_{k,l}}: \End(W) \to \End(M_{k,l})$ la repr\'esentation standard et
 \[ \rho^{\mathcal{B}}_{ M_{k,l}}: \Mat_2(\RR) \to \Mat_{\dim (M_{k,l})}(\RR)\]
 l'application induit pour des coordonn\'ees
de $W$ et $M_{k,l}$ fix\'ees. Par la d\'emonstration de la proposition \ref{prop_bilinear_decomp} resp. par la remarque \ref{rem_algorithme}, on peut supposer que les coordonn\'ees sont choisies de fa\c{c}on que $p_l \circ B_{\beta^k}= \rho^{\mathcal{B}}_{ M_{k,l}}(g_\beta)$ pour tout $\beta \in \N^1(A \times A)$. Comme les $\beta^k$ engendrent $\N^k(A^e)$, l'application $p_l \circ h$ est donc surjective si et seulement si
 les  matrices
\[ \rho^{\mathcal{B}}_{ M_{k,l}}(g_\beta)=\det(g_\beta)^{2l} \cdot \rho^{\mathcal{B}}_{ M_{k,l}}(g_\beta)\]
engendrent un espace vectoriel de dimension  $\frac{(k-2l+1)(k-2l)}{2}$.
Or les applications $ \rho^{\mathcal{B}}_{ M_{k,l}}$
sont polynomiales et l'ensemble
$U:=\{\beta \mid \det(\beta) \not = 0\}$ est un ouvert dans $\Symm_2(\RR)$, de sorte que
\[\langle \rho^{\mathcal{B}}_{M_{k,l}}(U) \rangle=\langle  \rho^{\mathcal{B}}_{M_{k-2l,0}}(U) \rangle.\] 
Il suffit ainsi de montrer que $\rho^{\mathcal{B}}_{M_{k,0}}(U)$ engendre un espace vectoriel de dimension 
$\frac{(k+2)(k+1)}{2}$ pour tout $k \in \NN$, ce qui est la conclusion du lemme \ref{lemme_sym_mat}.

\end{proof}

Soit $\rho:\End(W) \to \End(\Sym^k W)$ la repr\'esentation standard et 
$\rho_{\mathcal{B}}:\Mat_2(\RR) \to \Mat_{k+1}(\RR)$ l'application induite pour des coordonn\'ees associ\'ees \`a une base $\{w_1, w_2\}$ de $W$ et la base $\{w_1^i w_2^{k-i} \mid 0 \le i \le k\}$ de $\Sym^k W$.  
Soit $D \in \Mat_{k+1}(\RR)$ la matrice diagonale telle que $D_{ii}=\binom{n}{i}$ pour $0 \le i \le k$. 
Par la d\'emonstration du lemme \ref{lemme_sym_mat}, on sait que les matrices $\rho_{\mathcal{B}}(g)D$ sont sym\'etriques, de sorte que l'on obtient un isomorphisme
$b':\N^k(A \times A) \to \Symm_{k+1}(\RR)$ induit par la condition
$\beta^k \mapsto \rho_{\mathcal{B}}(g_\beta)D$ pour tout $\beta \in \N^1(A \times A)$ (cf. lemme \ref{equality_maps}).

\begin{lemme}
\label{lemme_mat_M}
Soit $A$ une vari\'et\'e ab\'elienne principalement polaris\'ee tr\`es g\'en\'erale de dimension $n$. Soit $1 \le k \le n$ et soit
 $\alpha \in \N^k(A \times A)$.
Avec les notations ci-dessus, $b_\alpha'$ repr\'esente $p_0 \circ B_\alpha=B_\alpha|_{\Sym^k W}$ pour tout $\alpha \in \N^k(A \times A)$
et l'on a 
\[
b_{\theta_1 \cdot \alpha}'=\left( \begin{array}{c | c}
 	b_\alpha'   & 0  \\ \hline
 	0						& 0
	\end{array} \right)
\]
pour  $1 \le k \le n-1$.
\end{lemme}

\begin{proof}
On a un plongement
canonique
\begin{eqnarray*}
\varphi: \Sym^k W & \to & W^{\otimes k}  \\
v_1 \cdots v_k & \mapsto & \sum_{\sigma \in S_k} v_{\sigma(1)} \otimes \dots \otimes v_{\sigma(k)}.
\end{eqnarray*}
Comme chaque tenseur dans $\varphi(v_1^i v_2^{k-i})$ appara\^it
$i! (k-i)!$ fois, et comme on a $\binom{k}{i}$ tenseurs, les vecteurs
\[\frac{1}{i! (k-i)! \sqrt{\binom{k}{i}} } \varphi(v_1^i v_2^{k-i}) , \, \, 0 \le i \le k,\] 
forment une base orthonorm\'ee de $\Sym^k W$  comme sous-module de $\bigwedge^k V_{\RR}$ par rapport
au produit int\'erieur standard. 
Soit $a: \Sym^k W \to \Sym^k W$ le changement de base donn\'e par la matrice
\[ a_{ij}= \begin{cases} c_0 i!(k-i)! \sqrt{\binom{k}{i}} & \text{si } i=j, \\
                                     0   & \text{sinon, }			
           \end{cases}\]
o\`u $c_0 \in \RR$.
La matrice $\rho_{\mathcal{B}'}(g)$ repr\'esentant $\rho(g)$ dans la nouvelle base $\mathcal{B}'$ est alors donn\'ee par
$\rho_{\mathcal{B}'}(g)=a \rho_{\mathcal{B}}(g) a^{-1}$. La matrice $\rho_{\mathcal{B}'}(g)$ ne d\'epend pas 
de la constante $c_0$ de sorte que  $\rho_{\mathcal{B}'}(g)$ est la matrice repr\'esentant
 $\rho(g)$ dans une base orthonorm\'ee de $\Sym^k W$. En particulier, on a
\[a\rho_{\mathcal{B}}(g_\beta)a^{-1}=\rho_{\mathcal{B}'}(g_\beta)= b_{\beta^k}|_{\Sym^k W}\]
pour tout $\beta \in \N^1(A \times A)$ et donc
$\rho_{\mathcal{B}}(g_\beta) = a^{-1} b_{\beta^k}|_{\Sym^k W} a$. Si l'on prend $c_0=\frac{1}{k!}$,
on obtient ainsi
\[\rho_{\mathcal{B}}(g_\beta) D= \transpose (a^{-1}) b_{\beta^k}|_{\Sym^k W}  a^{-1},\]
ce qui fournit la premi\`ere partie de l'\'enonc\'e.
Posons
\[b_\alpha'= \transpose (a^{-1}) b_\alpha|_{\Sym^k W} a^{-1}\]
et notons $(b_\alpha')_{ij}$ le coefficient correspondant aux coordonn\'ees $i,j$ dans la matrice
$b_\alpha'$ pour $0 \le i, j \le k$.
On peut calculer ces matrices en calculant les $b_{\beta^k}'$ et en appliquant 
ensuite la m\'ethode d\'ecrite dans la remarque \ref{rem_algorithme}.
De (\ref{equation_diagonal}) on obtient ainsi
\[(b_{\theta_1 \cdot \alpha}')_{i j}=\sum_{0 \le l \le k-i \atop j-i \le l \le j}  \binom{i-j+2l}{l} x_{k-i-l, j-l, i-j+2l}(\theta_1 \cdot \alpha)\]
pour $0 \le i,j \le k$ et
\[(b_\alpha')_{i j}=\sum_{0 \le l \le k-i \atop j-i \le l \le j } \binom{i-j+2l}{l} x_{k-1-i-l, j-l, i-j+2l}(\alpha)\]
pour $0 \le i,j \le k-1$.
Or, on a 
\[ x_{k-l-i-1, j-l, i-j+2l}(\alpha)=x_{k-i-l, j-l, i-j+2l}(\theta_1 \cdot \alpha)\]
pour $0 \le l \le k-i-1$,
\[x_{0, j-k+i, 2k-i-j}(\theta_1 \cdot \alpha)=0\]
et
\[(b_{\theta_1 \cdot \alpha}')_{ik}=(b_{\theta_1 \cdot \alpha}')_{ki}=0,\]
ce qui fournit le r\'esultat souhait\'e.
\end{proof}

En suivant 
la proc\'edure d\'ecrite dans la remarque \ref{rem_algorithme}, on calcule maintenant
des repr\'esentations  des c\^ones $\Semi^k(A \times A)$ pour $k=2,3,4$ pour une vari\'et\'e
ab\'elienne $A$ principalement polaris\'ee tr\`es g\'en\'erale de dimension $3$. On l'explique en 
d\'etail pour $k=2$.

\begin{enumerate}

\item 
 {\bfseries Le cas $k=2$.} 
Remarquons d'abord que l'on a
\[\bigwedge^2 W^{\oplus 3}=(\Sym^2 W)^{\oplus 3} \oplus \det(W)^{\oplus 6}.\]
Fixons des coordonn\'ees sur $W$ telles que l'action
de $g=\begin{pmatrix} a & b \\ c& d \end{pmatrix}$ sur $\Sym^2 W$ et $\det(W)$ soit donn\'ee par
\[\rho_{\det(W)}(g)=ad-bc \hskip 5mm \rho_{\Sym^2 W}(g)=\begin{pmatrix} a^2 & ab & b^2 \\ 2ac & ad+bc& 2bd \\ c^2 &  cd & d^2 \end{pmatrix}.\]
On a 
\[(a\theta_1+d\theta_2+b\lambda)^2=a^2\theta_1+d^2\theta_2+b^2 \lambda^2 +2ad\theta_1\theta_2+2ab\theta_1\lambda+2db\theta_2\lambda\]
et $g_{a\theta_1+d\theta_2+b\lambda}=\begin{pmatrix} a & b \\ b& d \end{pmatrix}$. Par le lemme \ref{equality_maps}, la matrice repr\'esentant la forme bilin\'eaire sym\'etrique
$B_{(a\theta_1+d\theta_2+b\lambda)^2}$
est donc semblable \`a une  matrice diagonale par blocs compos\'ee de $6$ blocs de la forme $ad-b^2 $ 
et $3$ blocs de la forme 
\[\begin{pmatrix} a^2 & ab & b^2 \\ 2ab & ad+b^2& 2bd \\ b^2 &  bd & d^2 \end{pmatrix}.\]
Plus g\'en\'eralement, pour
\[\alpha=x_{2,0,0}  \theta_1^2 +x_{1,1,0}  \theta_1 \theta_2   + x_{0,2,0}  \theta_2^2     + x_{1,0,1}  \theta_1 \lambda    + x_{0,1,1}  \theta_2 \lambda    + x_{0,0,2}  \lambda^2 \in \N^2(A \times A),\]
 une matrice repr\'esentant $B_\alpha$ 
est semblable \`a une matrice diagonale par blocs compos\'ee de $6$ blocs de la forme
\begin{equation}
\label{k=2_first}
\frac{1}{2}x_{1,1,0} -x_{0,0,2}
\end{equation}
et $3$ blocs de la forme
\[\begin{pmatrix}
 x_{2,0,0} & \frac{1}{2}x_{1,0,1} & x_{0,0,2} \\
x_{1,0,1} & \frac{1}{2}x_{1,1,0}+x_{0,0,2} & x_{0,1,1} \\
 x_{0,0,2} & \frac{1}{2} x_{0,1,1} & x_{0,2,0} \\
\end{pmatrix},\]
ce qui permet de d\'eduire des in\'equations d\'efinissant $\Semi^2(A \times A)$.

Pour obtenir une matrice diagonale par blocs repr\'esentant $B_\alpha$, il suffit par le lemme \ref{lemme_mat_M} de remplacer
la matrice $\rho_{\Sym^2 W}(g)$ par la matrice $\rho_{\Sym^2 W}(g)D$ (o\`u $D$ est la matrice diagonale avec $D_{ii}=\binom{2}{i}$ pour $i=0, 1, 2$)
dans le raisonnement ci-dessus.
Une matrice $b_\alpha$ repr\'esentant $B_\alpha$ est alors donn\'ee par une matrice diagonale par blocs compos\'ee de $6$ blocs de la forme (\ref{k=2_first})
et $3$ blocs de la forme
\begin{equation}
\label{k=2_second}
b_\alpha|_{\Sym^2 W}=\begin{pmatrix}
 x_{2,0,0} & x_{1,0,1} & x_{0,0,2} \\
x_{1,0,1} & x_{1,1,0}+2x_{0,0,2} & x_{0,1,1} \\
 x_{0,0,2} & x_{0,1,1} & x_{0,2,0} \\
\end{pmatrix},
\end{equation}
et la classe $\alpha$ est semipositive si et seulement si les matrices (\ref{k=2_first}) et (\ref{k=2_second}) sont semipositives.
On retrouve ainsi la repr\'esentation du c\^one $\Semi^2(A \times A)$ d\'ej\`a obtenue dans
\cite{DELV} par un calcul explicite.

\item 
{\bfseries Le cas $k=3$.}  Ayant
\[\bigwedge^3 W^{\oplus 3} \simeq (\det(W) \otimes W)^{\oplus 8} \oplus \Sym^3 W,\]
la classe
\[\alpha=\sum_{|l|=3} x_l \theta_1^{l_1}\theta_2^{l_2} \lambda^{l_3} \in \N^3(A \times A)=\Sym^3\N^1(A \times A)\]
 est semipositive si et seulement si les matrices
\begin{equation}
\label{k=3_first}
b_\alpha|_{\det(W) \otimes W}= \begin{pmatrix}
2x_{2,1,0}-2x_{1,0,2} &  x_{1,1,1}-6x_{0,0,3}  \\
x_{1,1,1}-6x_{0,0,3} & 2x_{1,2,0} -2x_{0,1,2} \\
\end{pmatrix}
\end{equation}
et 
\begin{equation}
\label{k=3_second}
b_\alpha|_{\Sym^3 W}=\begin{pmatrix} 
x_{3,0,0}     & x_{2,0,1}                   &  x_{1,0,2}                 &x_{0,0,3}                   \\
x_{2,0,1}     &    x_{2,1,0}+2x_{1,0,2}           & x_{1,1,1}+3x_{0,0,3}           &   x_{0,1,2}        \\
x_{1,0,2}     & x_{1,1,1}+3x_{0,0,3}            &  x_{1,2,0}+2x_{0,1,2}            &        x_{0,2,1}   \\
 x_{0,0,3}  & x_{0,1,2}                   &   x_{0,2,1}                &  x_{0,3,0}       \\
\end{pmatrix}
\end{equation}
sont semipositives.

\item 
 {\bfseries Le cas $k=4$.} 

Comme les seuls facteurs irr\'eductibles de $\bigwedge^4 W^{\oplus 3}$ sont $\det(W)^{\otimes 2}$ et $\det(W) \otimes \Sym^2 W$, 
un polyn\^ome
\[P(\theta_1, \theta_2,\lambda)=\sum_{|l|=4} x_l \theta_1^{l_1}\theta_2^{l_2} \lambda^{l_3} \in \Sym^4\N^1(A \times A)\]
repr\'esente une classe
semipositive $\alpha \in \N^4(A \times A)$ si et seulement si les matrices
\begin{equation}
\label{k=4_first}
b_\alpha|_{\det(W)^{\otimes 2}}=x_{2,2,0}-x_{1,1,2}+6x_{0,0,4}
\end{equation}
et
\begin{equation}
\label{k=4_second}
b_\alpha|_{\det(W) \otimes \Sym^2 W}=
					\begin{pmatrix} 
3x_{3,1,0}-2x_{2,0,2}     & 2x_{2,1,1}-6x_{1,0,3}        &  x_{1,1,2}-12x_{0,0,4}                    \\
2x_{2,1,1}-6x_{1,0,3}     &    4x_{2,2,0}-24x_{0,0,4}           & 2x_{1,2,1}-6x_{0,1,3}             \\
x_{1,1,2}-12x_{0,0,4}       & 2x_{1,2,1}-6x_{0,1,3}           &  3x_{1,3,0}-2x_{0,2,2}          \\
\end{pmatrix}
\end{equation}
sont semipositives. 
\end{enumerate}

\begin{rem}
Pour $k=2$ (resp. pour $k=3$), on obtient ainsi aussi une repr\'esentation
du c\^one $\Semi^2(A \times A)$ (resp. de $\Semi^3(A \times A)$) pour une vari\'et\'e
ab\'elienne principalement polaris\'ee tr\`es g\'en\'erale $A$ de dimension $\ge 2$ (resp. pour
$A$ de dimension $\ge 3$) par le corollaire \ref{corol_ind_dim}.
\end{rem}

\section{Comparaison des c\^ones}

Pour la suite, on suppose toujours que $A$ est une vari\'et\'e ab\'elienne principalement polaris\'ee tr\`es g\'en\'erale de dimension $n$.

\subsection{Classes pseudoeffectives et classes semipositives}
\label{section_classes_comp1}

Dans cette section, on \'etudie  les inclusions de c\^ones
\begin{equation}
\label{inclusions_resume}
\Sym^k \Psef^1(A^e) \subset \Psef^k(A^e)\subset \Strong^k(A^e)  \subset \Semi^k(A^e)
\end{equation}
en se servant des caract\'erisations explicites des c\^ones
$\Sym^k\Psef^1(A^e)$ (en fonction de g\'en\'erateurs) et $\Semi^k(A^e)$ 
(en fonction d'in\'equations le d\'efinissant).
Le r\'esultat principal est, qu'en codimension $3 \le k \le n$ et pour $e \ge 2$, on a (th\'eor\`eme \ref{thm_inclusion_stricte_1})
\[\Sym^k \Psef^1(A^e) \varsubsetneq \Semi^k(A^e).\] 

Par ailleurs, on montre que les inclusions (\ref{inclusions_resume})
sont des \'egalit\'es pour $n=3$ et $k=4$ (proposition \ref{prop_semi4}).

\begin{lemme}
\label{lemme_pullback_semi}
Soit $1 \le l \le e$ et soit $\alpha \in \N^k(A^l)$. Si $p:A^e \to A^l$ est une projection,
\begin{enumerate}
 \item pour $1 \le k \le ln$, la classe $p^* \alpha$ est semipositive si et seulement si $\alpha$ est semipositive, 
  \item pour $1 \le k \le n$, la classe $p^* \alpha$ est dans le c\^one $\Sym^k \Psef^1(A^e)$ si et seulement si $\alpha \in \Sym^k \Psef^1(A^l)$.
\end{enumerate}
\end{lemme}

\begin{proof}
La projection $p:A^e \to A^l$ correspond \`a une projection $\tilde p:U^{\oplus e} \to U^{\oplus l}$ qui 
induit naturellement une application $\tilde p_k: \bigwedge^k U^{\oplus e} \to \bigwedge^k U^{\oplus l}$.
Avec ces notations, on a $H_{\tilde p_k^* \alpha}=\tilde p_k^*H_\alpha$, et comme 
$\tilde p_k$ est surjectif, $\tilde p_k^*H_\alpha$ est semipositive si et seulement
si $H_\alpha$ est semipositive, ce qui montre la premi\`ere partie de l'\'enonc\'e.

Montrons la deuxi\`eme partie du lemme. On peut supposer que $p$ est la projection sur les $l$ premiers
facteurs de $A^e$. Si $\alpha \in \Sym^k \Psef^1(A \times A)$, il est
clair que l'on a aussi $p^*\alpha \in \Sym^k \Psef^1(A^e)$. Supposons donc $p^*\alpha \in \Sym^k \Psef^1(A^e)$,
i.e., on peut \'ecrire
\[p^*\alpha= \sum_{j=1}^l g_1^j\theta_1 \cdots g_k^j \theta_1\]
avec $g_i^j \in \GL(W) , 1 \le i \le k, 1 \le j \le l$.
Ecrivons $g_i^j \theta_1 =\sum_{r=1}^e (a_r^{i,j})^2 \theta_r + \sum_{1 \le s < t \le e} a_s^{i,j} a_t^{i,j}\lambda_{st}$,
o\`u $a_r^{i,j} \in \RR$ pour $1 \le r \le e$.
Comme $p^*\alpha$ est un polyn\^ome en $\theta_r, \lambda_{st},1 \le r \le l , 1 \le s<t \le l$, on a
$a_r^{i,j}=0$ pour $r \ge l$ et tout $i,j$. Cela entra\^ine $\alpha \in \Sym^k \Psef^1(A^e)$ et
donc le r\'esultat souhait\'e.
\end{proof}

\begin{lemme}
\label{lemme_inclusions} Soit $B$ une vari\'et\'e ab\'elienne. Alors on a
 les inclusions suivantes:
\begin{enumerate}[(a)] \rm
 \item  $\Semi^k(B) \cdot \Semi^l(B) \subset \Semi^{k+l}(B)$,
   \label{item1}
 \item $\Strong^k(B) \cdot \Strong^l(B) \subset \Strong^{k+l}(B)$,
    \label{item2}
 \item $\Nef^k(B) \cdot \Psef^l(B) \subset \Nef^{k+l}(B)$.
  \label{item4}
\end{enumerate}
\end{lemme}

\begin{proof}
Cela d\'ecoule de la d\'efinition des notions de positivit\'e respectives.
\end{proof}

Rappelons quelques propri\'et\'es \'el\'ementaires du c\^one $\Symm_k^+(\RR)$ des matrices $k \times k$ r\'eelles sym\'etriques
positives  \cite[II.12]{Barvinok}.
\begin{prop} 
\label{prop_cone_semi}
 Soit $\Symm_k^+(\RR)$ le c\^one  des matrices $k \times k$ r\'eelles sym\'etriques positives.
Alors on a
\begin{enumerate}
\item $\extr(\Symm_k^+(\RR))    =   \{M \in \Symm_k^+(\RR) \mid \rg(M) =1 \}$.
\item Pour toute matrice $A$ de rang $r<k$, il existe une (unique) face $F$ de
$\Symm_k^+(\RR)$ telle que $A$ est dans l'int\'erieur relatif de $F$. De plus, il existe une isom\'etrie
$F \simeq \Symm_r^+(\RR)$ preservant le rang des matrices dans $F$.
\end{enumerate}
\end{prop}

\begin{thm}
\label{thm_inclusion_stricte_1}
Soit $A$ une vari\'et\'e ab\'elienne principalement polaris\'ee tr\`es g\'en\'erale de dimension $n\ge 3$. Alors on a 
\begin{equation}
\label{strict_inclusion1}
\Sym^k \Psef^1(A^e) \varsubsetneq \Semi^k(A^e)
\end{equation}
pour $3 \le k \le n$ et pour $e \ge 2$. Sous les m\^emes restrictions sur $k$ et $n$, on a de plus
\begin{equation}
\label{inclusion_extrays}
\ext(\Sym^k \Psef^1(A\times A)) \subset \ext(\Semi^k(A\times A)).
\end{equation}
\end{thm}

\begin{proof}
Montrons d'abord (\ref{strict_inclusion1}) pour $e=2$.
Soit 
\[\alpha =\theta_1^2 + \theta_1 \theta_2 + \theta_2^2 +\lambda^2.\] 
Par la repr\'esentation explicite du
c\^one $\Semi^2(A \times A)$ donn\'ee dans (\ref{k=2_first}) et (\ref{k=2_second}), cette classe n'est pas semipositive, et ayant $\Sym^2 \Psef^1(A \times A)=\Semi^2(A \times A)$ (cf. \cite[Thm. 4.1]{DELV}), cela veut dire
qu'elle n'est pas contenue dans $\Sym^2 \Psef^1(A \times A)$. Par la proposition \ref{prop_im_inv}, on a 
donc 
\[\theta_1^{n-2} \cdot \alpha \not \in \Sym^n \Psef^1(A \times A)\]
 alors que l'on a par la repr\'esentation
explicite du c\^one $\Semi^3(A \times A)$ donn\'ee dans (\ref{k=3_first}) et (\ref{k=3_second}), et par le lemme \ref{lemme_inclusions}, 
\[\theta_1^{n-2} \cdot \alpha  \in \Semi^n(A \times A),\]
ce qui montre (\ref{strict_inclusion1}) pour $e=2$. On en d\'eduit 
(\ref{strict_inclusion1}) pour le cas g\'en\'eral $e\ge 3$ en tirant en arri\`ere les c\^ones respectivements par
des projections et en appliquant le lemme \ref{lemme_pullback_semi}.

Afin de montrer (\ref{inclusion_extrays}), on rappelle d'abord que la r\'eunion des rayons extr\'emaux de $\Sym^k \Psef^1(A\times A)$
est donn\'ee par (proposition \ref{prop_ext_rays})
\[E_k=\{g_1 \theta_1 \cdots g_k \theta_1 \mid g_1 ,\dots, g_k \in \GL(W)\} \subset \N^k(A \times A).\]
Par l'isomorphisme d\'ecrit dans le lemme \ref{lemme_mat_M} et par la proposition \ref{prop_cone_semi}, il suffit de montrer que
la matrice $b_\alpha'$ (comme d\'efinie dans le lemme \ref{lemme_mat_M}) est de rang $1$ pour tout 
$\alpha \in  E_k$.
Pour $k=1$, c'est la proposition \ref{corol_ext_1} de sorte que l'on peut supposer
$k \ge 2$.   Comme $E_k$ est invariant sous l'action de $\GL(W)$,
on peut supposer 
$\alpha= \theta_1 \cdot g_2 \theta_1 \cdots g_k \theta_1$. 
 Par r\'ecurrence, on
a $\rg(b_{g_2 \theta_1 \cdots g_k \theta_1}')=1$ et par le lemme \ref{lemme_mat_M},
cela fournit $\rg(b_\alpha')=1$, d'o\`u le r\'esultat.
\end{proof}

Par \cite{DELV}, on sait que l'on a $\Sym^{2n-2}(A \times A)=\Strong^{2n-2}(A \times A)$. La proposition suivante compl\`ete ce r\'esultat pour
$n=3$ et elle fournit ainsi des in\'equations d\'efinissant $\Psef^4(A \times A)$ dans ce cas. 

\begin{prop}
\label{prop_semi4}
Soit $A$ une vari\'et\'e ab\'elienne principalement polaris\'ee tr\`es g\'en\'erale de dimension $3$. Alors on a
\[\Sym^4 \Psef^1=\Psef^4(A \times A)=\Strong^4(A\times A)=\Semi^4(A \times A).\]
\end{prop}

\begin{proof}
Posons
\[L_{\ge 0}=\{\alpha \in \N^4(A \times A) \mid x_{2,2,0}-x_{1,1,2}+6x_{0,0,4} \ge 0\}\]
et
\[L=\{\alpha \in \N^4(A \times A) \mid x_{2,2,0}-x_{1,1,2}+6x_{0,0,4}= 0\}.\]
On montre $\Sym^4 \Psef^1(A \times A)=\Semi^4(A \times A)$.
Par la repr\'esentation explicite du c\^one $\Semi^4(A \times A)$ donn\'ee dans (\ref{k=4_first}) et (\ref{k=4_second}), 
ce c\^one est isomorphe \`a  l'intersection du c\^one
$\Symm_3^+(\RR)$ avec le demi-espace $L_{\ge 0}$.
Il en d\'ecoule qu'une classe  semipositive $\alpha$ dans $\N^4(A \times A)$ est extr\'emale dans 
 $\Semi^4(A \times A)$ si et seulement si  $\alpha$ est semipositive et la matrice $b_\alpha':=b_\alpha|_{\det(W) \otimes \Sym^2 W}$ est de rang $1$:
 si $b_\alpha'$ est de rang $1$, il est clair
que $\alpha$ est une classe extr\'emale par la proposition \ref{prop_cone_semi}. Inversement, supposons maintenant 
$\rg(b_{\alpha}') \not =1$. Si $\rg(b_\alpha')=2$, la matrice
$b_\alpha'$ appartient \`a l'int\'erieur relative d'une face $F$ de $\Symm_3^+(\RR)$ qui est isomorphe \`a $\Symm_2^+(\RR)$
(cf. proposition \ref{prop_cone_semi}).
On voit tout de suite que $b_\alpha'$ ne peut pas \^etre extr\'emale dans $F \cap L_{\ge 0}$ et donc pas dans 
$\Semi^4(A \times A)$. Si la matrice $b_\alpha'$ est de rang $3$, elle est dans l'int\'erieur du c\^one $\Symm_3^+(\RR)$. 
Il existe donc un voisinage $U$ dans l'int\'erieur de $\Symm_3^+(\RR)$ contenant $b_\alpha'$. Il s'ensuit que l'on peut \'ecrire 
$b_\alpha'=M_1+M_2$ avec $M_1, M_2 \in \Symm_3^+(\RR) \cap L_{\ge0} \cap U$ et $M_1 \not = M_2$, de sorte que
$b_\alpha'$ n'est pas extr\'emale. 

Il suffit ainsi de montrer que toute matrice de rang $1$ contenue dans $\Symm_3^+(\RR) \cap L_{\ge 0}$ repr\'esente une classe
dans $\Sym^4 \Psef^1(A\times A)$. On montre qu'une matrice de rang $1$ est repr\'esent\'ee
soit par une classe $g(\theta_1^2 \theta_2^2)$ soit par $g(\theta_1^3 \theta_2)$ pour un $g \in \GL(W)$. 
Remarquons qu'une matrice sym\'etrique de rang $1$ est enti\`erement d\'et\'ermin\'ee par sa premi\`ere colonne si
celle-ci est non nulle.

Montrons d'abord que toute matrice $b_\alpha'$ de rang $1$ dans $\Symm_3^+(\RR) \cap L_{>0}$ correspond \`a une
classe $g(\theta_1^2 \theta_2^2)$.  
Soit $g=\begin{pmatrix}   a & b \\ c & d \end{pmatrix}$.
Supposons d'abord $c=d=1$ et regardons
\[g(\theta_1^2 \theta_2^2)=(\theta_1 +a^2 \theta_2+a \lambda)^2 (\theta_1 +b^2 \theta_2+b \lambda)^2.\]
La matrice $b_\alpha'$ repr\'esente donc une telle classe $g(\theta_1^2 \theta_2^2)$ si et seulement si les \'equations suivantes,
que l'on obtient comme conditions sur la premi\`ere colonne, sont satisfaites:
\begin{align*}
3x_{3,1,0}-2x_{2,0,2}  & =   4(a-b)^2, \\
2x_{2,1,1}-6x_{1,0,3}  & =     4(a-b)^2(a+b), \\
x_{1,1,2}-12x_{0,0,4}  & =    4(a-b)^2ab.
\end{align*}
Comme $a \not=b$, on est ramen\'e aux \'equations
\begin{align*}
3x_{3,1,0}-2x_{2,0,2}  & =   1, \\
2x_{2,1,1}-6x_{1,0,3}  & =    a+b, \\
x_{1,1,2}-12x_{0,0,4}  & =   ab.
\end{align*}
Autrement dit, la matrice $b_\alpha'$ repr\'esente une classe $g(\theta_1^2 \theta_2^2)$
si et seulement si ce syst\`eme admet une solution r\'eelle.
C'est \'equivalent \`a dire que le polyn\^ome 
\begin{equation}
\label{polynome}
P(y)=(x_{1,1,2}-12x_{0,0,4}) -(2x_{2,1,1}-6x_{1,0,3})y+y^2
\end{equation}
admet deux racines r\'eelles distinctes (car $a \not = b$). Comme la matrice est de
rang $1$, on a 
\[(2x_{2,1,1}-6x_{1,0,3})^2=4x_{2,2,0}-24x_{0,0,4}\]
et le discriminant de $P$ vaut
\[\Delta(P)=x_{2,2,0}-x_{1,1,2}+6x_{0,0,4}.\]
Ainsi on voit que toute matrice de rang $1$ avec $3x_{3,1,0}-2x_{2,0,2} \not =0$, 
qui est dans $L_{>0}$, est aussi dans $\Sym^4 \Psef^1(A \times A)$.
De la m\^eme fa\c{c}on, on montre le r\'esultat pour les matrices $b_\alpha'$ de rang $1$ telles que $3x_{1,3,0}-2x_{0,2,2} \not = 0$. Si 
\[3x_{1,3,0}-2x_{0,2,2} = 3x_{3,1,0}-2x_{2,0,2} =0,\] 
et la matrice est de rang $1$, elle repr\'esente  un 
multiple de $\theta_1^2 \theta_2^2$.

Il reste donc \`a montrer que les matrices semipositives de rang $1$ dans $L$ correspondent aux
classes $g(\theta_1^3\theta_2)$. Or une matrice de rang $1$ correspond \`a une classe $g(\theta_1^3\theta_2)$
si et seulement si les \'equations suivantes, obtenues comme conditions sur la premi\`ere colonne, sont satisfaites:
\begin{align*}
3x_{3,0,0}-2x_{2,0,2}  & =   1, \\
2x_{2,1,1}-6x_{1,0,3}  & =   2a, \\
x_{1,1,2}-12x_{0,0,4}  & =   a^2,
\end{align*}
ce qui est le cas si et seulement si  le polyn\^ome 
$P(y)$ d\'efini dans (\ref{polynome})
admet une racine double r\'eelle $a$. Par un raisonnement comme
pour $g(\theta_1^2 \theta_2^2)$, on trouve que c'est le cas si et seulement si la classe est contenue dans $L$, ce qui
fournit le r\'esultat souhait\'e. 
\end{proof}

\begin{rem}
Comme tous les modules irr\'eductibles dans une d\'ecomposition de $\bigwedge^{2n-2} W^{\oplus n}$ sont isomorphes \`a
$\det(W)^{\otimes n-1}$ ou \`a $\det(W)^{\otimes n-2} \otimes \Sym^2 W$, le c\^one $\Semi^{2n-2}(A \times A)$
peut toujours \^etre identifi\'e avec le c\^one des matrices sym\'etriques r\'eelles semipositives $3 \times 3$ intersect\'e
avec un demi-espace (cf. proposition \ref{prop_bilinear_decomp}). Pour des raisons comme dans la d\'emonstration de la proposition \ref{prop_semi4},
on devrait  avoir 
\[\Sym^{2n-2} \Psef^1(A \times A)=\Psef^{2n-2}(A \times A)=\Strong^{2n-2}(A \times A)=\Semi^{2n-2}(A \times A)\]
pour tout $n \ge 3$, et les rayons extr\'emaux de ce c\^one devraient correspondre aux classes $\GL(W)\cdot(\theta_1^n \theta_2^{n-2})$ et $\GL(W)\cdot (\theta_1^{n-1} \theta_2^{n-1})$.
Mais pour l'instant, je ne vois pas de moyen pour montrer  le cas g\'en\'eral.
\end{rem}

Avec un raisonnement semblable \`a celui utilis\'e dans la d\'emonstration de la proposition \ref{prop_semi4},
on peut montrer que le c\^one $\Sym^3 \Psef^1(A \times A)$ 
s'identifie avec le c\^one engendr\'e par les matrices de rang $1$ dans la repr\'esentation du
c\^one $\Semi^3(A \times A)$ donn\'ee dans (\ref{k=3_first}) et (\ref{k=3_second}).

\begin{question}
Si l'on regarde $\Semi^k(A \times A)$ comme un sous-c\^one de $\Symm_{k+1}^+(\RR)$, est-ce que,
pour $1\le k \le n$, l'ensemble des rayons extr\'emaux de $\Sym^k \Psef^1(A \times A)$ (resp. de $\Sym^{2n-k} \Psef^1(A \times A)$) s'identifie 
 avec l'ensemble des matrices semipositives de rang $1$ dans $\Semi^k(A \times A)$?
\end{question}

\subsection{Classes num\'eriquement effectives et classes pseudoeffectives}
\label{section_classes_nefs}

Dans \cite{DELV}, les auteurs montrent que, pour une surface ab\'elienne $A$,
la classe $\mu=4\theta_1 \theta_2-\lambda^2$ est nef mais pas pseudoeffective
 de sorte que l'on a une inclusion stricte 
$\Psef^2(A \times A) \varsubsetneq \Nef^2(A \times A)$.
Le r\'esultat principal de cette section (proposition \ref{thm_inclusion_stricte_2}) est que l'on a, en toute dimension $n$,
\begin{equation*}
\Psef^k(A^e) \varsubsetneq \Nef^k(A^e)
\end{equation*}
pour tout entier positif $e \ge 2$  pour $2 \le k \le ne-2$.

Rappelons d'abord qu'en tant que $\GL(W)$-modules, on a
\begin{equation}
\label{young_decomp_2}
\Sym^k \N^1(A \times A)=\bigoplus_{0 \le 2i \le 2k} \mu^{i} \cdot \Sym^{2k-4i}W \oplus \RR_+\mu^{n-k}
\end{equation}
et $g \cdot \mu=\det(g)^2 \mu$ pour tout $g \in \GL(W)$.

\begin{lemme}
\label{mu_nef}
La classe $\mu^k$ est nef pour 
$k \in \{1, \dots, n\}$. En particulier, on a
\[\mu^k \cdot \Psef^{n-k}(A \times A) \subset \Nef^{n+k}(A \times A).\]
\end{lemme}

\begin{proof}
Par la d\'ecomposition (\ref{young_decomp_2})
et la proposition \ref{intersections}, il suffit de montrer que
le c\^one  $\Psef^{2n-2k}(A \times A)$ est contenu dans le demi-espace
\[H^+_{\mu^{n-k}}=\bigoplus_{0 \le 2i \le 2n-2k} \mu^{i} \cdot \Sym^{4(n-k)-4i}W \oplus \RR_+\mu^{n-k}.\]
Ayant
$ \Psef^{2n-2k}(A \times A) \subset \Semi^{2n-2k}(A \times A)$, il suffit de montrer
que l'on a $\Semi^{2n-2k}(A \times A) \subset H^+_{\mu^{n-k }}$.
Pour une classe $\alpha \in \N^{2n-2k}(A \times A)$, la matrice $b_\alpha$ repr\'esentant $B_\alpha$ se 
d\'ecompose en blocs correspondant \`a une  d\'ecomposition 
de $\bigwedge^{2n-2k} V_{\RR}$ en $\GL(W)$-modules irr\'eductibles (proposition \ref{prop_bilinear_decomp}). 
Comme $\det(W)^{\otimes n-k}$ est un module irr\'eductible apparaissant
dans une telle d\'ecomposition de $\bigwedge^{2n-2k} V_{\RR}$, on a des blocs $1 \times 1$
dans la matrice $b_\alpha$ correspondant au module irr\'eductible 
\\ $\Sym^2(\det(W)^{n-k})\simeq \RR \mu^{n-k}$ pour
tout $\alpha \in \N^{2n-2k}(A \times A)$. 
Si l'on \'ecrit $P_\alpha=\sum_{0 \le 2i \le 2n-2k} \mu^i P_i(\theta_1, \theta_2, \lambda)$ avec
$P_i \in \Sym^{4(n-k)-4i}W$, selon la d\'ecomposition
donn\'ee dans la proposition \ref{corol_decomp_N}, alors $P_{n-k}$ est une constante et 
l'application 
\begin{align*}
\RR\mu^{n-k} & \to \Sym^2(\det(W)^{n-k}) \\
P_{n-k} &\mapsto L(\alpha)
\end{align*}
 \'etant la multiplication par une constante r\'eelle $c \not =0$,
on a $L(\alpha)=cP_{n-k}$. Si $c > 0$, on a donc $\Semi^{2n-2k}(A \times A) \subset H^+_{\mu^{n-k}}$ et si $c<0$, on a
$\Semi^{2n-2k}(A \times A) \subset H^-_{\mu^{n-k}}$. Comme la classe $\theta_1^{n-k}\theta_2^{n-k}$
est semipositive, il suffit ainsi de montrer
\[\theta_1^{n-k}\theta_2^{n-k} \in H^+_{\mu^{n-k}}.\]
 Ecrivons $\theta_1^{n-k}\theta_2^{n-k}=\sum_{0 \le 2i \le 2n-2k} \mu^i P_i$ avec
$P_i \in \Sym^{4(n-k)-4i}W$, selon la d\'ecomposition
(\ref{young_decomp_2}). Alors $P_{n-k}$ est une constante et on veut d\'eterminer son signe.
Comme le morphisme surjectif $\Sym^{2n-2k}\N^1(A \times A) \to \N^{2n-2k}(A \times A)$ correspond \`a une projection sur des
facteurs irr\'eductibles de $\Sym^{2n-2k} \N^1(A \times A)$, et comme $\mu^{n-k}$ engendre un module irr\'eductible
 non nul dans $\N^{2n-2k}(A \times A)$ pour $k \le n$, on peut supposer que $A$ est de dimension $n-k$.
Or on a $\N^{2n-2k}(A \times A)=\RR\mu^{n-k}$, et $\mu^{n-k} P_{n-k}$ est donc juste un nombre d'intersection de
$2n-2k$ diviseurs effectifs, de sorte que l'on a $\mu^{n-k} P_{n-k} \ge 0$. 
Comme $\mu^{n-k} >0$ par la remarque \ref{intersections}, on obtient $P_{n-k} >0$, ce qui fournit le r\'esultat souhait\'e.
\end{proof}

\begin{thm}
\label{thm_inclusion_stricte_2}
Soit $A$ une vari\'et\'e ab\'elienne principalement polaris\'ee tr\`es g\'en\'erale de dimension $n \ge 2$. 
Les classes $\theta_1^k \mu \in \N^{k+2}(A \times A)$ et $\theta_1^{n-2}\theta_2^k\mu \in \N^{n+k}(A \times A)$
ne sont pas semipositives mais nefs pour $0 \le k \le n-2$, et l'on a
\begin{equation}
\label{eqn_psefinnef}
\Psef^k(A^e) \varsubsetneq \Nef^k(A^e)
\end{equation}
pour tout entier positif $e \ge 2$ et $2 \le k \le ne-2$.
\end{thm}

\begin{proof}
 Par la proposition \ref{mu_nef} il est clair 
que $\theta_1^k \mu $ et $\theta_1^{n-2}\theta_2^k\mu$ sont  des classes nefs 
pour $0 \le k \le n-2$.
Pour voir que $\theta_1^k \mu$ n'est pas semipositive pour $0 \le k \le n$, il suffit de remarquer que
la matrice $b_{\mu}'$ (cf. lemme \ref{lemme_mat_M}) n'est pas semipositive, ce qui entra\^ine par 
le lemme \ref{lemme_mat_M} que
$b_{\theta_1^k\cdot \mu}'$ n'est pas semipositive,
et donc $\theta_1^k \cdot \mu  \not \in \Semi^k(A \times A)$.

Pour voir que $\theta_2^k \theta_1^{n-2} \mu=4\theta_2^{k+1} \theta_1^{n-1}-\theta_2^k \theta_1^{n-2}\lambda^2$ n'est pas semipositive, 
on regarde la matrice $h_{\theta_2^k \theta_1^{n-2} \mu}$ repr\'esentant $H_{\theta_2^k \theta_1^{n-2} \mu}$
dans la base standard de $\bigwedge^k V$.
On montre qu'il y a un $2 \times 2$ mineur principal dont le d\'eterminant est n\'egatif.
La matrice $h_{\theta_2^{k+1} \theta_1^{n-1}}$ est une matrice avec des coefficients z\'eros hors
de la diagonale et des coefficients non z\'eros dans la diagonale pour les coordonn\'ees
$z_{i_1} \wedge \dots \wedge z_{i_{n-1}} \wedge z_{j_1} \wedge \dots \wedge z_{j_{k+1}}$ avec
$i_1, \dots, i_{n-1} \in \{1, \dots, n\}$ et $j_1, \dots, j_{k+1} \in \{n+1, \dots, 2n\}$. 
En m\^eme temps la matrice $h_{\theta_2^k \theta_1^{n-2}\lambda^2}$ contient un coefficient non nul
pour
\[z_{n+1} \wedge i \bar z_{n+1} \wedge \dots \wedge z_{n+k} \wedge i \bar z_{n+k} \wedge z_{1} 
	\wedge  i \bar z_{1} \wedge \dots \wedge z_{n-2} \wedge  i \bar z_{n-2}  \wedge z_{n+(n-1)} \wedge i \bar z_{n-1} \wedge z_{2n} \wedge i \bar z_n,\]
	qui n'est pas sur la diagonale, et comme on n'a pas de coefficients de $\theta_2^k \theta_1^{n-2}\lambda^2$ dans 
	la diagonale pour
\[z_{n+1} \wedge i \bar z_{n+1} \wedge \dots \wedge z_{n+k} \wedge i \bar z_{n+k} \wedge z_{1} 
	\wedge  i \bar z_{1} \wedge \dots \wedge z_{n} \wedge  i \bar z_{n} ,\]
cela entra\^ine le r\'esultat.

Montrons maintenant l'inclusion (\ref{eqn_psefinnef}). Par les arguments pr\'ec\'edents, on a
\[\Semi^k(A \times A) \varsubsetneq \Nef^k(A \times A)\]
 pour $2 \le k \le 2n-2$. Cela nous permet de raisonner par
r\'ecurrence sur $e$ pour $n$ fix\'e, en supposant que l'\'enonc\'e est vrai pour $e-1$. De plus, on peut se restreindre \`a le montrer
pour $2 \le k \le \lfloor \frac{ne}{2} \rfloor$ par dualit\'e.
Soit $\alpha \in \Nef^k(A^{e-1})$ une classe
nef non semipositive. Alors $p_{1, \dots, e-1}^*\alpha$ est nef et non semipositive
par le lemme \ref{lemme_pullback_semi}. Ayant $ne-2 \ge \lfloor \frac{ne}{2} \rfloor$ pour $n\ge 2, e\ge 2$, cela ach\`eve la d\'emonstration.
\end{proof}

\begin{rem}
Par  \cite[Prop. 3.2]{DELV}, on a des isomorphismes
$\cdot \mu^k: \N^{n-k}(A \times A) \to \N^{n+k}(A \times A)$ pour $1 \le k \le n$, et on se demande naturellement si les c\^ones de classes positives
respectifs sont pr\'eserv\'es, ce qui a \'et\'e v\'erifi\'e dans \cite{DELV} pour $k=n-1$. Par le th\'eor\`eme \ref{thm_inclusion_stricte_2}, 
on voit que, pour $n=3$, l'isomorphisme $\cdot \mu: \N^2(A \times A) \to \N^4(A \times A)$
ne pr\'eservent pas les classes pseudoeffectives, de sorte qu'en g\'en\'eral, on 
 ne peut pas s'attendre \`a ce que les c\^ones soient
pr\'eserv\'es.
\end{rem}

\begin{rem}
On montre \'egalement que les in\'equations d\'efinissant $\Nef^{2n-2}(A  \times A)$ ne d\'ependent pas 
de $n$, si l'on \'ecrit les in\'equations en fonction des coordonn\'ees associ\'ees \`a la base
\[\{\mu^{n-2} \cdot \theta_1^2, \mu^{n-2} \cdot \theta_1 \theta_2, \mu^{n-2} \cdot \theta_2^2, 
\mu^{n-2} \cdot \theta_1 \lambda, \mu^{n-2} \cdot \theta_2 \lambda, \mu^{n-2} \cdot \lambda^2\}\]
de $\N^{2n-2}(A \times A)$. Comme les in\'equations d\'efinissant $\Nef^2(A  \times A)$ ont 
\'et\'e calcul\'ees dans \cite{DELV} pour $n=2$, on obtient ainsi des in\'equations explicites d\'efinissant $\Nef^{2n-2}(A \times A)$
pour tout $n \ge 2$.

D'autre c\^ot\'e, les in\'equations d\'efinissant $\Nef^2(A \times A )$ d\'ependent de $n$ alors que
le c\^one $\Psef^2(A \times A) =\Semi^2(A \times A)$ ne d\'epend pas de $n$ pour $n \ge 2$.
\end{rem}

\begin{rem}
Tous les r\'esultats obtenus  pour une vari\'et\'e ab\'elienne principalement polaris\'ee tr\`es g\'en\'erale $A$ concernant la structure alg\'ebrique de $\N^\bullet(A^e)$ et
les c\^ones dans $\N^k(A^e)$
sont \'egalement vrais pour une vari\'et\'e ab\'elienne $A$ principalement polaris\'ee quelconque  
 si l'on se 
restreint \`a la $\RR$-alg\`ebre $\N^\bullet_{\can}(A^e) \subset \N^\bullet(A^e)$ engendr\'ee par
les $\theta_i$ et les $\lambda_{j,k}, 0 \le i \le e, 1\le j  <k \le e$.

Remarquons de plus qu'une isog\'enie $f:B \to B'$ entre deux vari\'et\'es ab\'eliennes induit un isomorphisme $f^*:\N^\bullet(B') \to \N^\bullet(B)$
qui pr\'eserve les c\^ones en question \cite[Prop. 1.6]{DELV}. Les r\'esultats pour $\N^\bullet(A^e)$ pour $A$ principalement polaris\'ee tr\`es g\'en\'erale sont
donc \'egalement vrais pour $\N^\bullet(B)$ si $B$ est isog\`ene \`a $A^e$.
\end{rem}

\addcontentsline{toc}{section}{R\'ef\'erences}

\bibliography{ab_eff}

\begin{thebibliography}{10}

\bibitem{Abeasis}
S.~Abeasis.
\newblock The {GL(V)}-invariant ideals in {$S(S^2(V))$}.
\newblock {\em Rend. Mat.}, (6) 13:235--262, 1980.

\bibitem{Barvinok}
A.~Barvinok.
\newblock {\em A course in convexity}, volume~54 of {\em Grad. {S}tud. {M}ath.}
\newblock American {M}athematical {S}ociety, 2000.

\bibitem{Demailly}
J.~P. Demailly.
\newblock {\em Complex {A}nalytic and {D}ifferential {G}eometry}.
\newblock 2009.

\bibitem{Hazama}
F.~Hazama.
\newblock Algebraic cycles on certain abelian varieties and powers of special
  surfaces.
\newblock {\em J. {F}ac. {S}ci. {U}niv. {T}okyo {S}ect. {IA} {M}ath},
  31:2487--520, 1985.

\bibitem{howe}
R.~Howe.
\newblock Remarks on classical invariant theory.
\newblock {\em Trans. {A}mer. {M}ath. {S}oc.}, 313(2):239--570, 1989.

\bibitem{HK}
R.~Harvey{,}~Q.W. Knapp.
\newblock Positive (p,p) forms, {W}irtinger's inequality, and currents.
\newblock In {\em Value-distribution theory, Part A}, pages 43--62, 1974.

\bibitem{BL}
C.~Birkenhake{,}~H. Lange.
\newblock {\em Complex {A}belian {V}arieties}, volume 302 of {\em Grundlehren
  {M}ath. Wiss.}
\newblock Springer, 2003.

\bibitem{Lawson}
B.~Lawson.
\newblock The stable homology of a flat torus.
\newblock {\em Math. {S}cand.}, 36:49--73, 1975.

\bibitem{DELV}
O.~Debarre{,} L. Ein{,}~R. {L}azarsfeld {et} C.~Voisin.
\newblock Pseudoeffective and nef classes on abelian varieties.
\newblock {\em {C}ompos. {M}ath. to appear}.

\bibitem{PS}
A.~Prendergast-Smith.
\newblock The cone conjecture for abelian varieties.
\newblock {\em preprint}, 2010.
\newblock arxiv:1008.4509v1 [math.AG].

\bibitem{Ribet}
K.~Ribet.
\newblock Hodge classes on certain types of abelian varieties.
\newblock {\em {A}mer. {J}. {M}ath.}, 105:523--538.

\bibitem{Tambour}
T.~Tambour.
\newblock A note on some representations of {$SL(V)$}.
\newblock {\em Indag. Math.}, 6 (4):505--509, 1995.

\bibitem{Tankeev}
S.~G. Tankeev.
\newblock Cycles on simple abelian varieties of prime dimension (in russian).
\newblock {\em Izv. {A}kad. {N}auk {USSR} {S}er. {M}at.}, 46:155 -- 170.

\bibitem{Thompson}
G.~Thompson.
\newblock Skew invariant theory of symplectic groups, pluri-{H}odge groups and
  3-manifold invariants.
\newblock {\em Int. {M}ath. {R}es. {N}ot.}, 10.1093/imrn/rnm048, 2007.

\bibitem{Weyman}
J.M. Weyman.
\newblock {\em Cohomology of {V}ector {B}undles and {S}yzygies}, volume 149 of
  {\em Cambridge {T}racts in {M}ath.}
\newblock Cambridge {U}niversity {P}ress, 2003.

\end{thebibliography}
\bibliographystyle{plain}

\end{document}